\documentclass[final,leqno]{siamart1116}
\usepackage{tikz,pgfplots,pgfplotstable}
\usepackage{amsmath,amssymb}
\usepackage{subfigure}
\usepackage{geometry}
\usepackage{multirow}
\usepackage{algorithm}
\usepackage{longtable}
\usepackage[algo2e,linesnumbered,vlined,ruled]{algorithm2e}
\usepackage{enumerate}

\newtheorem{example}{Example}
\newtheorem{assumption}{Assumption}

\newcommand{\argmin}{\operatornamewithlimits{argmin}}
\newcommand{\iprod}[2]{\langle #1, #2 \rangle}
\newcommand{\diag}{\mathrm{diag}}

\numberwithin{equation}{section}
\numberwithin{example}{section}

\newcounter{cnt}
\foreach \num in {1,2,...,26}{%
  \stepcounter{cnt}%
  \expandafter\xdef \csname c\Alph{cnt}\endcsname {\noexpand\mathcal{\Alph{cnt}}}%
  \expandafter\xdef \csname b\Alph{cnt}\endcsname {\noexpand\mathbb{\Alph{cnt}}}%
}
\newcommand{\be}{\begin{equation}}
\newcommand{\ee}{\end{equation}}
\newcommand{\bee}{\begin{equation*}}
\newcommand{\eee}{\end{equation*}}
\newcommand{\bea}{\begin{eqnarray}}
\newcommand{\eea}{\end{eqnarray}}
\newcommand{\beaa}{\begin{eqnarray*}}
\newcommand{\eeaa}{\end{eqnarray*}}

\newcommand{\st}{\mbox{ s.t. }}

\newcommand{\rng}{\mathrm{Range}}
\newcommand{\tr}{\mathrm{tr}}

\newcommand{\orth}{ \mathrm{orth} }
\newcommand{\prox}{ \mathrm{prox} }
\newcommand{\dist}{ \mathrm{dist} }

\graphicspath{ {./fig/} }
\title{Low-rank Matrix Optimization Using Polynomial-filtered Subspace Extraction}
\author{
	Yongfeng Li \thanks{Beijing International Center for Mathematical Research, Peking University, CHINA (YongfengLi@pku.edu.cn)}
	\and
	Haoyang Liu \thanks{Beijing International Center for Mathematical Research, Peking University (liuhaoyang@pku.edu.cn).}
	\and
	Zaiwen Wen
	\thanks{Beijing International Center for Mathematical
		Research, Peking University (wenzw@pku.edu.cn).
		Research supported in part by the NSFC grants 11421101 and 11831002, and by the National
		Basic Research Project under the grant 2015CB856002.}
	\and
	Yaxiang Yuan
	\thanks{State Key Laboratory of Scientific and Engineering Computing, Academy of Mathematics and Systems Science,Chinese Academy of Sciences, China (yyx@lsec.cc.ac.cn). Research supported in part by NSFC grants 11331012 and 11461161005.}
}
\begin{document}
\maketitle

\begin{abstract}
    In this paper, we study first-order methods on a large variety of
low-rank matrix optimization problems, whose solutions only live in a low
dimensional eigenspace. Traditional first-order methods depend on the eigenvalue decomposition 
at each iteration which takes most of the computation time. 
In order to reduce the cost, we propose an inexact algorithm framework
based on a polynomial subspace extraction.
The idea is to use an additional polynomial-filtered iteration to extract an approximated eigenspace, 
and project the iteration matrix on this subspace, followed by an optimization update.
The accuracy of the extracted subspace can be controlled by the degree of the polynomial filters.
This kind of subspace extraction also enjoys the warm start property: the subspace of the
current iteration is refined from the previous one. 
Then this framework is instantiated into two algorithms: 
the polynomial-filtered proximal gradient method and the
polynomial-filtered alternating direction method of multipliers.
We give a theoretical guarantee to the two algorithms that the polynomial degree is not necessarily very large.
They share the same convergence speed as the corresponding original methods if
the polynomial degree grows with an order $\Omega(\log k)$ at the $k$-th iteration.
If the warm-start property is considered, the degree can be reduced to a constant, independent of the
iteration $k$. Preliminary numerical experiments on
several low-rank matrix optimization problems show that the polynomial filtered algorithms usually provide multi-fold speedups.

\end{abstract}


\section{Introduction} \label{sec:intro}
Eigenvalue decompositions (EVD) are commonly used in large varieties of
matrix optimization problems with spectral or low-rank structures.
For example, in semi-definite programmings (SDP), many optimization methods need to compute all positive eigenvalues
and corresponding eigenvectors of a matrix each iteration
in order to preserve the semi-definite structure.
In matrix recovery type problems, people are interested in the low-rank
approximations of the unknown matrix. In \cite{NIPS2009_3704}, the authors
propose a convex relaxation model to solve the robust PCA problem and solve
it by the accelerated proximal gradient (APG) approach. The primal problem is
solved by the augmented Lagrangian method (ALM) in \cite{Lin2013TheAL}.
The application can also be regarded as an optimization problem on the
rank-$r$ matrix space thus manifold optimization methods can be applied \cite{low-rank-riemannian}.
In all methods
mentioned above, EVD of singular value decompositions (SVD) are required. The latter can be
transformed to the former essentially.
In maximal eigenvalue problems, the objective function is the largest eigenvalue
of a symmetric matrix variable. It is used in many real applications,
 for example, the dual formulation of
the max-cut problem \cite{goemans1995improved}, phase retrieval, 
blind deconvolution \cite{doi:10.1137/15M1034283}, distance metric
learning problem \cite{ying2012distance}, and Crawford number computing \cite{kressner2018subspace}.
These optimization methods require the largest
eigenvalue and the corresponding eigenvector at each iteration. The variable dimension is usually large.

In general, at least one full or truncated EVD per iteration is required
in most first-order methods to solve these applications.
It has been long realized that EVD is
very time-consuming for large problems, especially when
the dimension of the matrix and the number of the required eigenvalues are
both huge. First-order methods suffer from this issue greatly since they usually take
thousands of iterations to converge. Therefore, people turn to inexact methods
to save the computation time.
One popular approach relies on approximated eigen-solvers to solve the eigenvalue sub-problem, 
such as the Lanczos method and randomized EVD with early stopping rules, see \cite{zhou2011godec,becker2013randomized,soltani2017fast}
and the references therein. The performance of these methods is determined by
the accuracy of the eigen-solver, which is sometimes hard to control in practice.
Another type of method is the so-called subspace method.
By introducing a low-dimensional subspace, one can greatly reduce the dimension
of the original problem, 
and then perform refinement on this subspace as the iterations proceed.
For instance, it is widely used to solve an univariate maximal 
eigenvalue optimization problem \cite{kressner2018subspace,sirkovic2016subspace,kangal2018subspace}.
In \cite{zhou2006self}, the authors propose an inexact method which simplifies
the eigenvalue computation using Chebyshev filters in self-consistent-field (SCF) calculations.
Though inexact methods are widely used in real applications,
the convergence analysis is still very limited. Moreover, designing a practical strategy
to control the inexactness of the algorithm is a big challenge.

\subsection{Our Contribution}
Our contributions are briefly summarized as follows.
\begin{enumerate}
	\item For low-rank matrix optimization problems involving eigenvalue computations,
	we propose a general inexact first-order method framework with polynomial
	filters which can be applied to most existing solvers without much difficulty.
	It can be observed that for low-rank problems, the iterates always lie in a low dimensional
	eigenspace. Hence our key motivation is to use
	one polynomial filter subspace extraction step to estimate the subspace, followed by
	a standard optimization update. The algorithm also benefits from the warm-start
	property: the updated iteration point can be fed to the polynomial filter again
	to generate the next estimation. Then we apply this framework to the
	proximal gradient method (PG) and the alternating direction method of multipliers (ADMM) to obtain
	a polynomial-filtered PG (PFPG) method and a polynomial-filtered ADMM (PFAM)
    method (see Section \ref{sec:pf-algs}), respectively.
	\item We analyze the convergence property of PFPG and PFAM. It can be proved
	that the error of one exact and inexact iteration is bounded by the principle angle
	between the true and extracted subspace, which is further controlled by the polynomial
	degree. The convergence relies on an assumption that the initial space should not
	be orthogonal to the target space, which is essential but usually holds in many applications.
	Consequently, the polynomial degree barely increases during the iterations. It can
	even remain a constant under the warm-start setting. This result provides us the opportunity
	to use low-degree polynomials throughout the iterations in practice.
	\item We propose a practical strategy to control the accuracy of the polynomial
	filters so that the convergence is guaranteed. A portable implementation
	of the polynomial filter framework is given based on this strategy. 
	It can be plugged into any first-order methods with only a few lines
	of codes added, which means that the computational cost reduction is almost a free
	lunch.
\end{enumerate}

We mention that our work
is different from randomized eigen-solver based methods \cite{zhou2011godec,becker2013randomized,soltani2017fast}.
The inexactness of our method is mainly dependent on the quality of the subspace,
which can be highly controllable using polynomial filters and the warm-start strategy. On the other hand, 
the convergence of randomized eigen-solver
based methods relies on the tail bound or the per vector error bound of the solution returned by the
eigen-solver, which is usually stronger than subspace assumptions.
Our work also differs from the subspace method proposed in \cite{kressner2018subspace,kangal2018subspace}.
First the authors mainly focus on maximal eigenvalue problems, which have special structures,
while we consider general matrix optimization problems with low-rank variables. Second, in
the subspace method, one has to compute the exact solution of the sub-problem, which is
induced by the projection of the original one onto the subspace. However, after we obtain
the subspace, the next step is to extract eigenvalues and eigenvectors from the projected matrix
to generate inexact variable updates. There are no sub-problems in general. 

\subsection{Source Codes}
The \textsc{matlab} codes for several low-rank matrix optimization examples
are available at \url{https://github.com/RyanBernX/PFOpt}.

\subsection{Organization}
The paper is organized as follows. In Section \ref{sec:pf-algs}, we introduce 
the polynomial filter algorithm and propose two polynomial-filtered optimization
methods, namely, PFPG and PFAM. Then the convergence analysis is established in Section \ref{sec:conv}.
Some important details of our proposed algorithms
in order to make them practical are summarized in Section \ref{sec:impls}.  The effectiveness of PFPG
and PFAM is demonstrated in Section \ref{sec:num} with a number of practical applications. Finally,
we conclude the paper in Section \ref{sec:conclusion}.
\subsection{Notation}
Let $\mathcal{S}^n$ be the collection of all $n$-by-$n$ symmetric matrices. For any $X \in \mathcal{S}^n$, we use $\lambda(X) \in \mathbb{R}^n$
to denote all eigenvalues of $X$, which are permuted in descending order, i.e.,
$\lambda(X) = (\lambda_1, \ldots, \lambda_n)^T$ where $\lambda_1 \ge \lambda_2 \cdots \ge \lambda_n$.
For any matrix $X \in \mathbb{R}^{n\times n}$, $\mathrm{diag}(X)$
denotes a vector in $\mathbb{R}^n$ consisting of all diagonal entries of
$X$. For any vector $x \in \mathbb{R}^n$, $\mathrm{Diag}(x)$ denotes
a diagonal matrix in $\mathbb{R}^{n\times n}$ whose $i$-th diagonal entry
is $x_i$. We use $\iprod{A}{B}$ to define the inner product in the matrix
space $\mathbb{R}^{n\times n}$, i.e., $\iprod{A}{B}:=\mathrm{Tr}(A^TB)$
where $\mathrm{Tr}(X):=\sum_{i=1}^n X_{ii}$ for $X \in \mathbb{R}^{n\times n}$.
The corresponding Frobenius norm is defined as $\|A\|_F = \mathrm{Tr}(A^TA)$.
The Hadamard product of two matrices
or vectors of the same dimension is denoted by $A\odot B$ with $(A\odot B)_{ij} = A_{ij}\times B_{ij}$.
For any matrix $X \in \bR^{n\times p}$, $\rng(X)$ denotes the subspace spanned by
the columns of $X$.

\section{Polynomial Filtered Matrix Optimization} \label{sec:pf-algs}
\subsection{Chebyshev-filtered Subspace Iteration} \label{sec:cheby}
The idea of polynomial filtering is originated from a well-known fact that
polynomials are able to manipulate the eigenvalues of any symmetric matrix $A$
while keeping its eigenvectors unchanged. 
Suppose $A$ has the eigenvalue decomposition $A = U\mathrm{Diag}(\lambda_1,\ldots,\lambda_n)U^T$,
then the matrix $\rho(A)$ has the eigenvalue decomposition
$\rho(A) = U\mathrm{Diag}(\rho(\lambda_1),\dots,\rho(\lambda_n))U^T$.

Consider the traditional subspace iteration, $U \leftarrow AU$, the convergence of the desired
eigen-subspace is determined by the gap of the eigenvalues, which can be
very slow if the gap is nearly zero. This is where the polynomial filters
come into practice: to manipulate the eigenvalue gap aiming to a
better convergence. For any polynomial $\rho(t)$ and initial matrix $U$, the polynomial-filtered 
subspace iteration is given by
\[
U \leftarrow \rho(A)U.
\]

In general, there are many choices of $\rho(t)$. One popular choice is to use
Chebyshev polynomials of the first kind, whose explicit expression can be
written as
\begin{equation} \label{eq:cheb}
T_d(t) = 
\begin{cases}
\cos(d\arccos t) & |t| \leq 1, \\
\frac{1}{2}( (t-\sqrt{t^2-1})^d + (t+\sqrt{t^2-1})^d ) & |t| > 1,
\end{cases}
\end{equation}
where $d$ is the degree of the polynomial.
One important property of Chebyshev polynomials is that they grow pretty fast outside
the interval $[-1,1]$,
which can be helpful to suppress all unwanted eigenvalues in this
interval effectively. Figure \ref{fig:cheby} shows some examples of Chebyshev
polynomials $T_d(t)$ with $d=1,\ldots,6$ and power of Chebyshev polynomials $T_3^q(t)$
with $q=1,\ldots,5$. That is exactly why Chebyshev polynomials are taken
into our consideration.

\begin{figure}[h]
	\centering
	\caption{Chebyshev polynomials and their variants} \label{fig:cheby}
	\subfigure[Chebyshev polynomials of the first kind $T_d(t)$]{
		\includegraphics[width=0.48\textwidth]{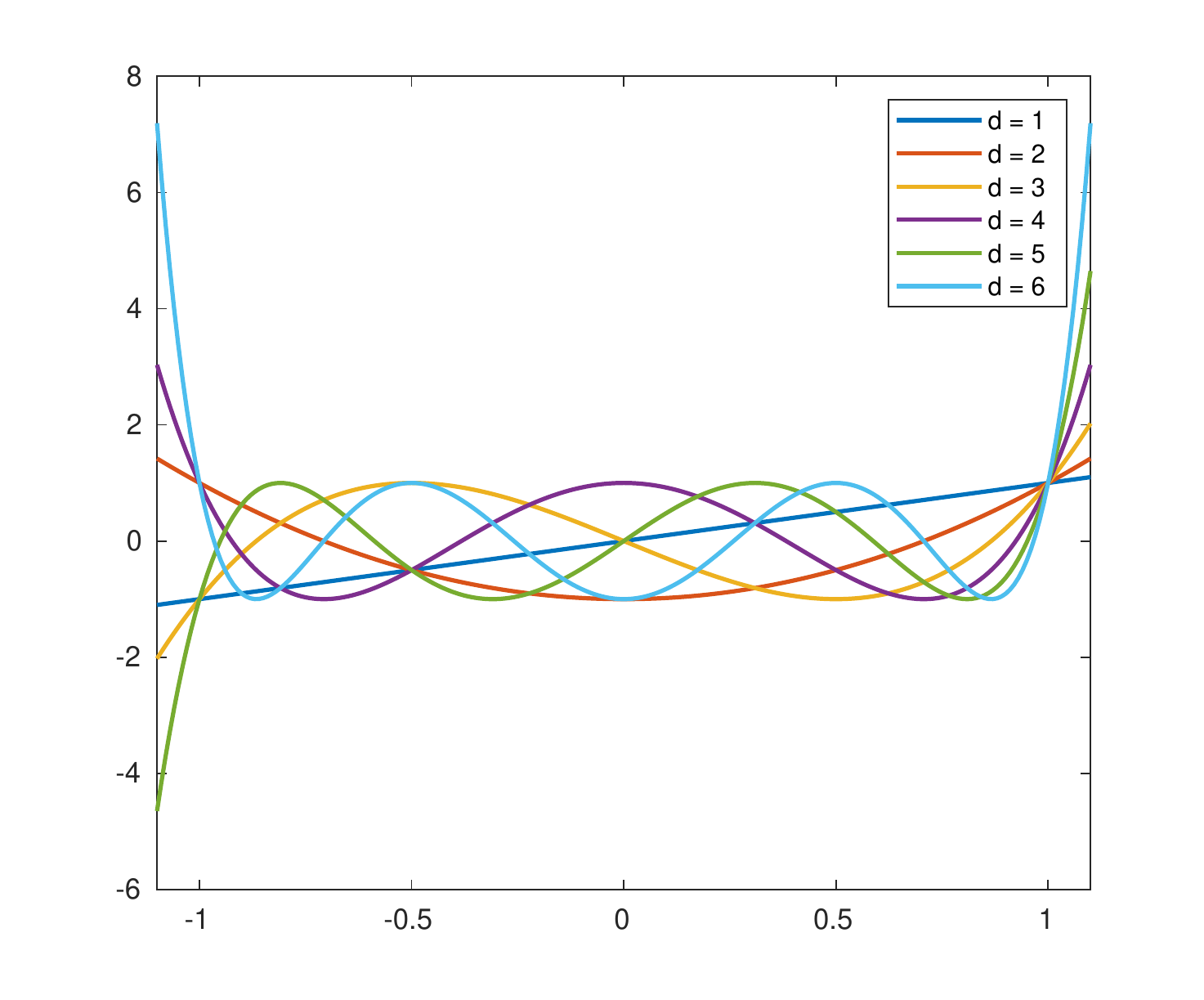}
	}
	\hfill
	\subfigure[Power of Chebyshev polynomial $T_3^q(t)$]{
		\includegraphics[width=0.48\textwidth]{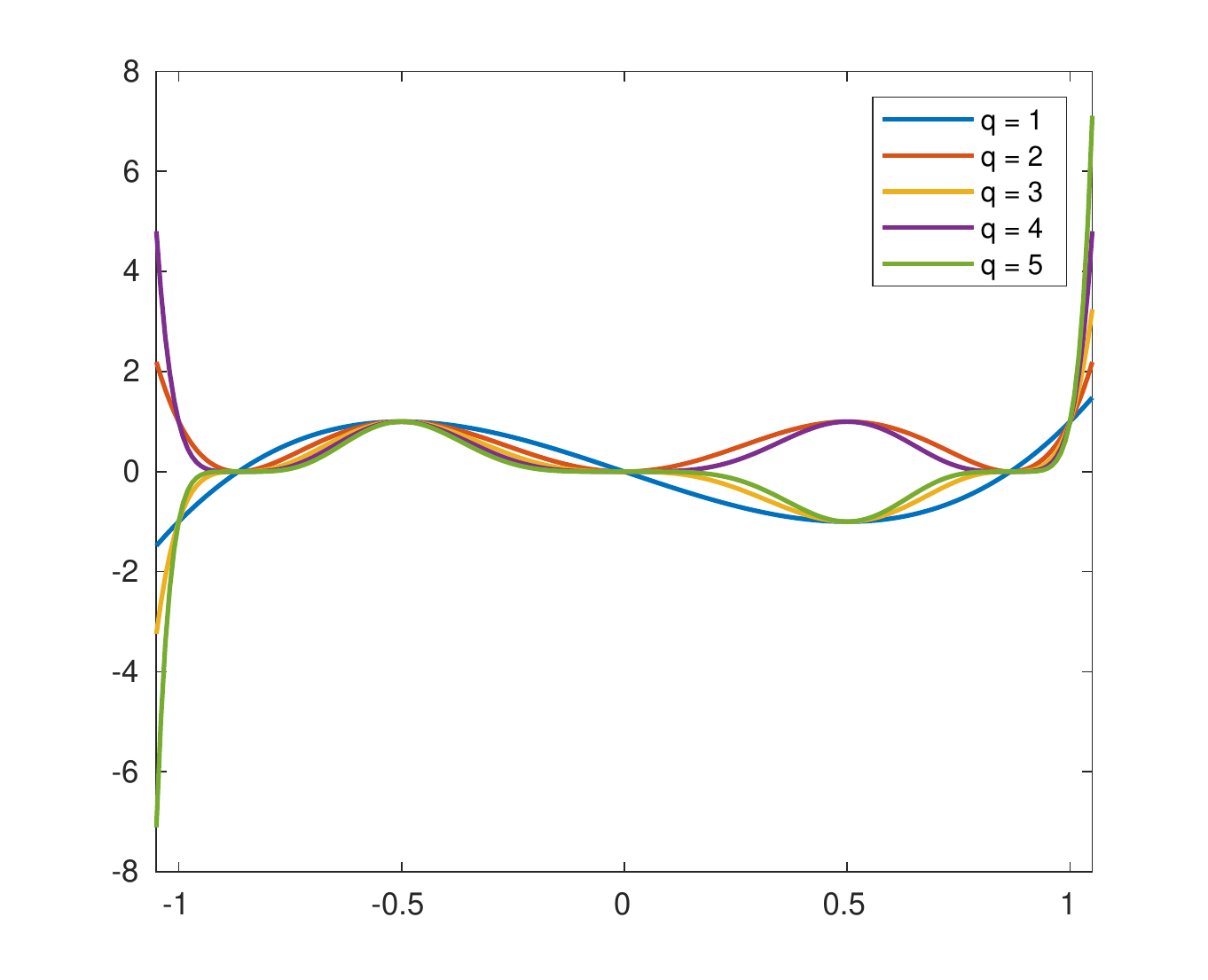}
	}
\end{figure}

In order to suppress all eigenvalues in a general interval $[a, b]$,
a linear mapping from $[a, b]$ to $[-1, 1]$ is constructed. To be precise,
the polynomial can be chosen as
\begin{equation} \label{eqn:cheb-linear-map}
\rho(t) = T_{d}\left(\frac{t-(b + a) / 2}{(b - a) / 2}\right).
\end{equation}

After the polynomial filters are applied, an orthogonalization step 
usually follows in order to prevent the subspace from losing rank. This
step is often performed by a single QR decomposition. Consequently, 
given an arbitrary matrix $U \in \bR^{n\times p}$ and a polynomial $\rho(t)$,
the polynomial-filtered subspace iteration can be simply written as
\begin{equation} \label{eqn:subspace-update}
	U^+ = \orth(\rho(A)U),
\end{equation}
where ``$\orth$'' stands for the orthogonalization operation.
In most cases, the updated matrix $U^+$ spans a subspace which contains some approximated
eigenvectors of $A$. By extracting eigen-pairs from this subspace, we are able to derive
a number of polynomial-filtered inexact optimization methods. In the next
subsections we present two examples: PFPG and PFAM.

%After the polynomial filters are applied to obtain an approximate 
%subspace, the next step is to use the Rayleigh-Ritz (RR) procedure
%to extract optimal eigenvalues and eigenvectors from this subspace.
%Combining these two steps yields the CheSE algorithm (Algorithm \ref{alg:CheSE})

%\begin{algorithm2e}[H]
%	\caption{Chebychev-filtered Subspace Extraction (CheSE)} \label{alg:CheSE}
%	Input: $A \in \mathcal{S}^{n}$, initial matrix $U_0\in \mathbb{R}^{n\times m}$, interval $[a, b]$. \\
%	Compute Chebyshev filtering $Z \leftarrow\rho(A)(U_0)$ (See Algorithm \ref{alg:Cheb-eval}). \\
%	Orthogonalize $Z$ using QR decomposition: $Z=QR$. \\
%	Compute the projection matrix: $M\leftarrow Q^TAQ$. \\
%	Compute the eigenvalue decomposition on projected matrix: $M = VDV^T$. \\
%	Set outputs: $U\leftarrow QV$, and $\Sigma = D$.\\
%\end{algorithm2e}
%
%The lines from 3 to 6 are actually the Rayleigh-Ritz procedure. We need to
%perform a full eigenvalue decomposition, but in a subspace $\cS^{m}$ which is much
%smaller . The main cost of Algorithm \ref{alg:CheSE} is the
%Chebyshev filtering and the QR decomposition.
\subsection{The PFPG Method} \label{sec:IPGM}
In this subsection we show how to apply the subspace update \eqref{eqn:subspace-update} to the proximal
gradient method, on a set of problems with a general form.
Consider the following unconstrained spectral operator minimization problem
\begin{equation} \label{eqn:mat-opt}
\begin{array}{rl}
    \min & h(x) := F(x) + R(x), \\
\end{array}
\end{equation}
where $F(x)$ has a composition form $F(x)=f\circ\lambda(\cB(x))$ with $\cB(x)
= G+\mathcal{A}^*(x)$ and $R(x)$ is a regularizer with simple structures but
need not be smooth.  Here $G$ is a known matrix in $\mathcal{S}^n$,
and $\mathcal{A}^*: \mathbb{R}^m\rightarrow \mathcal{S}^n$ is a linear
operator. The function $f: \mathbb{R}^n \rightarrow \mathbb{R}$ is smooth
and \emph{absolutely symmetric}, i.e., $f(x) = f(Px)$ for all $x \in
\mathbb{R}^n$ and any permutation matrix $P \in \mathbb{R}^{n\times n}$.  It
can be easily verified that $F(x)$ is well defined.

According to \cite[Section 6.7.2]{parikh2014proximal}, the gradient of $F$
in \eqref{eqn:mat-opt} is
\begin{equation} \label{eqn:grad-F}
\nabla F(x) = \mathcal{A}(\Psi(\cB(x))),
\end{equation}
where $\Psi$ is a spectral operator defined by
\[
    \Psi(X) = V\mathrm{Diag}(\nabla f (\lambda(X)))V^T,
\]
and $V\mathrm{Diag}(\lambda(X))V^T$ is the full eigen-decomposition of $X$. 

For any matrix $S\in \bR^{n\times p}$, let $\cP_{S}$ denote the projection operator on subspace $\rng(S)$. When $S$ 
is an orthogonal matrix, the operator $\cP_{S}$ is defined by
\begin{equation}
\cP_{S}(X) = SS^TXSS^T.
\end{equation}
If $V$ is exactly a $p$-dimensional eigen-space of the matrix $X$, the eigenvalue decomposition
of $\cP_V(X)$ can be written as
\begin{equation} \label{eqn:RR}
\cP_{V}(X) = \begin{bmatrix}
\tilde{V} & V_\perp
\end{bmatrix}
\begin{bmatrix}
\Lambda_p & O \\
O & O
\end{bmatrix}
\begin{bmatrix}
\tilde{V}^T \\
V_\perp^T
\end{bmatrix},
\end{equation}
where $(\tilde{V},\Lambda_p)$ is the corresponding eigen-pairs and we have $\rng(V) = \rng(\tilde{V})$.
Note that $(\tilde{V},\Lambda_p)$ has closed-form expression with $\tilde{V}=VY$ where $Y\Lambda_pY^T$ is
the full eigenvalue decomposition of $H:=V^TXV \in \bR^{p\times p}$.

%and define the noisy residual function by
%$
%r_{\tau}(x,\epsilon) = \frac{1}{\tau}(x - \prox_{\tau R}(x-\tau (\nabla F(x)+\epsilon)))
%$. 

Suppose
evaluating $\nabla F$ only involves a small number of eigenvalues in the problem \eqref{eqn:mat-opt},
then we are able to compute $\nabla F$ quickly as long as the corresponding 
subspace is given. We summarize them in Assumption \ref{assump:f}.

\begin{assumption} \label{assump:f}
\begin{enumerate}[(i)]
\item Let $\cI(x) \subset [n]$ be an integer set with $\cI(x) = \{s,s + 1,...,t\}$
and suppose that the gradient of $F(x)$ in \eqref{eqn:mat-opt} has the relationship
\begin{equation} \label{eqn:grad-F-2nd}
\nabla F(x) = \bar{g}:= \mathcal{A}(\Phi(\cP_{V_\cI}(\cB(x)))),
\end{equation}
where $\Phi$ is a spectral operator and
$V_\cI \in \mathbb{R}^{n\times |\cI|}$ contains eigenvectors  $v_i, i \in \cI(x)$ of $\cB(x)$. 

%\item $\cA$ are bounded with factor $L_{\cA}$, i.e., $\|\cA X\|_2 \leq L_{\cA}\|X\|_F$.
\item $\Phi$ is Lipschitz continuous with a factor $L_{\Phi}$,
\begin{equation}
\|\Phi(X) - \Phi(Y)\|_F \leq L_{\Phi}\|X-Y\|_F.
\end{equation}
\end{enumerate}
\end{assumption}

We now compare the gradient expression \eqref{eqn:grad-F-2nd} in Assumption \ref{assump:f} with
the original gradient \eqref{eqn:grad-F}. Expression \eqref{eqn:grad-F-2nd} implies that the subspace
$V_\cI$ contains all information for evaluating $\nabla F(x)$. In other words, we do not need
full eigenvalue decompositions as indicated by \eqref{eqn:grad-F}. The spectral operator $\Phi(\cdot)$ in \eqref{eqn:grad-F-2nd}
is also different from $\Psi(\cdot)$ defined in \eqref{eqn:grad-F}. The choice of $\Phi(\cdot)$ can be
arbitrary as long as it satisfies Assumption \ref{assump:f}. In most cases $\Phi(\cdot)$
inherits the definition of $\Psi(\cdot)$ but they need not be the same.
%
%We show a few examples of $f$ that satisfy such assumption.
%\begin{example}
%	Let
%	\[
%	f(\lambda) = \frac{1}{2}\sum_{i=1}^{n}\max(\lambda_i, 0)^2.
%	\]
%	It is easy to see $(\nabla f)_i=\max(\lambda_i, 0)$. For $\lambda$ with
%	only a few positive entries, $f$ satisfies this assumption.
%\end{example}
%
%\begin{example}
%	Let
%	\[
%	f(\lambda) = \log\sum_{i=1}^n\exp(\lambda_i),
%	\]
%	whose gradient is $(\nabla f)_i=\exp(\lambda_i)/\sum_{i=1}^{n}\exp(\lambda_i)$.
%	Although both $f$ and $\nabla f$ is dependent on all entries of $\lambda$,
%	if there is a large gap between $\lambda_1$ and $\lambda_2$, then $f$ and
%	$\nabla f$ can be approximated by dropping $\exp(\lambda_i), i \geq r$,
%	for some $r$. Then \eqref{eqn:assump-f} is satisfied for small $\delta$.
%\end{example}
%
%\begin{remark}
%The construction of such $\phi$ is highly dependent on
%the form of $f$. And $(\phi(\lambda_\alpha))_i, i \notin \alpha$
%must be zero because the gradient of $F$ still depends on the
%eigenvectors of $\lambda_i, i \notin \alpha$.
%\end{remark}
Under the Assumption \ref{assump:f}, we can write the proximal gradient algorithm as
\begin{equation} \label{eq:prox-update-1}
    x^{k+1} = \mathrm{prox}_{\tau_kR}(x^k - \tau_k \mathcal{A}(\Phi(\cP_{V^k_{\cI_k}}(\cB(x^k))),
\end{equation}
where $V^k_{\cI_k}$ contains the eigenvectors of $\cB(x^k)$ at iteration $k$, $\tau_k$ is the step size, and the proximal operator is defined by 
\begin{equation}
\prox_{tR}(x) = \argmin_{u} R(u) + \frac{1}{2t}\|u -x \|_2^2.
\end{equation}

In the scheme \eqref{eq:prox-update-1}, the main cost is the computation of
the EVD of $\cP_{V^k_{\cI_k}}(\cB(x^k))$ as the exact subspace $V^k_{\cI_k}$
is unknown. Thus, an idea to reduce the
expenses is to use polynomial filters to generate an approximation of
eigen-space. Then it is easy to compute the projection on this space and the
corresponding eigen-decomposition, whence the image of the spectral operator
on the projection matrix, as shown in \eqref{eqn:RR}. To obtain a good approximation, the update of
eigen-space is by means of Chebyshev polynomials which can suppress all
unwanted eigenvectors of $\cB(x)$ and the eigen-space in last step is used as initial
values. In summary, the proximal gradient method with polynomial filters can
be written as
\begin{eqnarray}
    x^{k+1} &=& \mathrm{prox}_{\tau_kR}(x^k - \tau_k \mathcal{A}(\Phi(\cP_{U^k}(\cB(x^k)))), \label{eq:pfpg1}\\
    U^{k+1} &=& \orth(\rho_{k+1}^{q_{k+1}}(\cB(x^{k+1}))U^k), \label{eq:pfpg2}
\end{eqnarray}
where $U^k \in \bR^{n \times p}$ with $p \geq |\cI_k|$ is an orthogonal matrix which
serves as an approximation of $V^k_{\cI_k}$ each step,
and $q_k \geq 1$ is a small integer (usually 1 to 3) which means the polynomial is applied for $q_k$ times
before the orthogonalization.

We mention that \eqref{eq:pfpg1} and \eqref{eq:pfpg2} are actually two different updates
that couple together. Given a subspace spanned by $U^k$, the step \eqref{eq:pfpg1}
performs the proximal gradient descent with the approximated gradient extracted from $U^k$.
The next step \eqref{eq:pfpg2} makes refinement on the subspace to obtain $U^{k+1}$ via
polynomial filters.
The subspace  will
become more and more exact as the iterations proceed. In practice, it is usually
observed that $U^k$ is very ``close'' to the true eigen-space $V^k_{\cI_k}$
during the last iterations of the algorithm.

%%%%%%%%%%%%%%%%%%%%%%%%%%%%%%%%%%%%%%%%%%%%%%%%%%%%%%%%%%%%%%%%%%%%%%%
%%%%%%%%%%%%%%%%%%%%%%%%%%%%%%%%%%%%%%%%%%%%%%%%%%%%%%%%%%%%%%%%%%%%%%%
\subsection{The PFAM Method} \label{sec:PFAM}
Consider the following standard SDP:
\begin{equation} \label{eq:sdp-primal}
\begin{array}{rl}
    \min & \iprod{C}{X}, \\
    \st & \cA X = b, X \succeq 0,
\end{array}
\end{equation}
where $\cA$ is a bounded linear operator.
Let $F(X) = 1_{\{X \succeq 0\}}(X)$ and $G(X) = 1_{\{\cA X =b \}}(X) +
\iprod{C}{X}$, where $1_{\Omega}(X)$ is the indicator function on a set
$\Omega$. Then the Douglas-Rachford Splitting (DRS) method on the primal SDP
\eqref{eq:sdp-primal} can be written as
\[
    Z^{k+1} = T_{\textrm{DRS}} (Z^k),
\]
where
\[
    T_{\textrm{DRS}} = \prox_{tG}(2\prox_{tF} - I) - \prox_{tF} + I,
\]
which is equivalent to the ADMM on the dual problem.
The explicit forms of $\prox_{tF}(Z)$ and $\prox_{tG}(Y)$ can be written as 
\begin{eqnarray*}
	\prox_{tF}(Z) &=& \cP_+(Z), \\
    \prox_{tG}(Y) &=& (Y + t C) - \cA^*(\cA\cA^*)^{-1}(\cA Y+ t\cA C-b),
\end{eqnarray*}
where $\cP_{+}(Z)$ is the projection operator onto the positive semi-definite cone.

Similar to the PFPG method, $\cP_+(Z)$ is only determined by the eigen-space spanned by the positive
eigenvectors of $Z$, so we are able to
use the polynomial filters to extract an approximation $U^k$. Therefore, the PFAM can be written as 
\begin{eqnarray}
    Z^{k+1} &=& \prox_{tG}(2\cP_+(\cP_{U^k}(Z^k)) - Z^k) - \cP_+(\cP_{U^k}(Z^k)) + Z^k, \label{eq:PFAM1}\\
    U^{k+1} &=& \orth(\rho_{k+1}^{q_{k+1}}(Z^{k+1})U^k), \label{eq:PFAM2}
\end{eqnarray}
where $U^k \in \bR^{n \times p}$ with $p \geq |\cI|$ is an orthogonal matrix and $q_k \geq 1$ is a small integer.

\section{Convergence Analysis} \label{sec:conv}

\subsection{Preliminary}

In this section we introduce some basic notations and tools that are used in our analysis. Firstly, we introduce the definition of principal angles to measure the distance between two subspaces.
\begin{definition}
    Let $X , Y \in \bR^{n \times p}$ be two orthogonal matrices. The singular values of $X^TY$ are $\sigma_1 \geq \sigma_2 \geq \cdots \geq \sigma_p$, which lie in $[0,1]$. The principal angle between $\rng(X)$ and $\rng(Y)$ is defined by
    \be
    \Theta(X, Y):= \mathrm{Diag}(\arccos \sigma_1,\ldots,\arccos \sigma_p),
    \ee
    In particular, we define $\sin \Theta$ by taking the sine of $\Theta$ componentwisely.
\end{definition}
%
%We have the following propositions.
%\begin{proposition}\label{prop:sin}
%    \begin{enumerate}[(i)]
%        \item
%        $
%        \|\sin \Theta(X,Y)\|_F = \|X^T_{\bot} Y\|_F = \|X^TY_{\bot}\|_F = \frac{1}{\sqrt{2}}\|XX^T - YY^T\|_F,
%        $
%        % \item 
%        %$\|\tan \Theta(X,Y)\|_F = \|X^T_{\bot} Y(X^TY)^{\dagger}\|_F,$
%    \end{enumerate}
%\end{proposition}
%\begin{proof}
%    The item (i) is shown in Theorem 2.5.1 in \cite{golub2012matrix}. 
%    %and the item (ii) is in Lemma 4.3 of \cite{Drineas2016}.
%\end{proof}
%

%   
%   Then we define the projection operator.
%   \begin{definition}\label{lem:cheb-est}
%       The $k$-dimensional projection operator on $\rng(U)$ are defined by
%       \be
%       \cP_{U,k}(A) = U\cdot(\underset{\rank(X)=k}{\arg\min}\|A-UXU^T\|_F)\cdot U^T,
%       \ee
%       where $k \leq \rank(U)$.
%   \end{definition}

The following lemma describes a very important property of Chebyshev polynomials.
\begin{lemma}\label{lem:cheb-est}
    The Chebyshev polynomials increase fast outside $[-1,1]$, i.e.,
\be
T_d(\pm (1+\epsilon)) \geq 2^{d \min \{\sqrt{\epsilon},1\}-1}, \ \forall \epsilon > 0. 
\ee
\end{lemma}
\begin{proof}
It follows from the expression \eqref{eq:cheb} that
\beaa
T_d(\pm (1+\epsilon)) &\geq& \frac{1}{2} (1 + \epsilon + \sqrt{2\epsilon + \epsilon^2})^d 
\geq \frac{1}{2}(1+\sqrt{\epsilon})^d \\
&=& \frac{1}{2}e^{d\log(1+\sqrt{\epsilon})} \geq  2^{d \min \{\sqrt{\epsilon},1\}-1},
\eeaa
where the last inequality follows from that $\log(1+x) \geq \log 2 \cdot\min\{ x, 1\}$. 
\end{proof}

\subsection{Convergence analysis for PFPG}

In this subsection, we analyze the convergence for general convex problems. The relative gap is defined by
\begin{equation} \label{eqn:relative-gap}
\cG_k = \min_{i \in \cI_k}\left(\frac{|\lambda_{i}- (a_k+b_k)/2|}{(a_k-b_k)/2} - 1\right),
\end{equation}
where $[a_k, b_k]$ is the interval suppressed by Chebyshev polynomials at the $k$-th iteration, $\cI_k$ is the index set of the eigenvalues beyond $[a_k, b_k]$.
Throughout this section, we make the following assumptions.
\begin{assumption}\label{assump:gcvx}
    \begin{enumerate}[(i)]
        \item $\|\sin \Theta(V^{k+1}_{\cI_{k+1}},U^k)\|_2 < \gamma, \ \forall \ k$ with $\gamma < 1$, where
        $V^k_{\cI_{k}}$ and $U^k$ are defined in Assumption \ref{assump:f} and \eqref{eq:pfpg2} at iteration $k$.
        %\item The noisy gradients are bounded, i.e., $\|\mathcal{A}(\Phi(U^k (U^k)^T(\cB(x^k))U^k (U^k)^T)\| \leq M_1$. 
        \item The sequence generated by iteration \eqref{eq:pfpg1} and \eqref{eq:pfpg2} are bounded, i.e.,
        $
        \|x^k\|_2 \leq C, \ \forall \ k. 
        $
        \item The relative gap \eqref{eqn:relative-gap} has a lower bound, i.e., $\cG_k \geq l, \ \forall \ k$. 
    \end{enumerate}
\end{assumption}
Assumption (i) implies that the initial eigen-space is not orthogonal to the
truely wanted eigen-space so that we can obtain $V^{k+1}_{\cI_{k+1}}$ by
performing subspace iterations on $U^k$. Essentially, this property is required by
almost all iterative eigensolvers for finding the correct eigenspace at each
iteration.   Assumption (ii) is commonly used in the optimization literature. Assumption (iii) is a key assumption which guarantees that the relative gap $\cG_k$ exists in the asymptotic sense. Therefore, the polynomial filters are able to separate the wanted
eigenvalues from the unwanted ones. In eigenvalue computation, this property is
often enforced by adding a sufficient number of guarding vectors so that the
 corresponding eigenvalue gap is large enough.

%\begin{proof}
%The iteration \eqref{eq:cheb} can be see one iteration in orthogonal iteration. Then it follows from Theorem 8.2.2 in \cite{golub2012matrix},
%\bee
%\|\sin\Theta(\bar{U}^{k+1}, U^{k+1})\|_2 \leq \eta_{k+1}^{q_{k+1}} \|\tan\Theta(\bar{U}^{k+1}, U^k)\|_2 \leq \frac{\eta_{k+1}^{q_k}\gamma}{\sqrt{1-\gamma^2}}, 
%\eee
%where the last inequality is due to (ii) in the Assumption \ref{assump:g}.
%\end{proof}

The following lemma gives an error estimation when using an inexact gradient of $F$.
It states that the polynomial degree $d_k$ should be proportional to
$\log(1/\varepsilon)$, where $\varepsilon$ is a given tolerance.
\begin{lemma}\label{lem:ctr1}
Suppose that Assumptions \ref{assump:f} and \ref{assump:gcvx} hold. Define the error between the exact and inexact gradients:
\be\label{eq:eps}
e^k = \mathcal{A}(\Phi(\cP_{U^k}(\cB(x^k))))
-\nabla F(x^k). 
\ee
Let the polynomial filter $\rho_k(t)$ be a Chebyshev polynomial with degree $d_k$.
To achieve $\|e^k\|_2 \leq \varepsilon$ for a given tolerance $0 < \varepsilon < 1$, the degree $d_k$ should satisfy  
\be\label{eq:poly-times}
    d_k = \Omega\left(\frac{\log \frac{1}{\varepsilon}}{\min\{l,1\}}\right).
\ee
\end{lemma}

\begin{proof}
%For convenience, we assume that $\alpha_k = 0$. 
The update \eqref{eq:cheb} can be regarded as one orthogonal
iteration. The eigenvalue of $\rho_k^{q_k}(\cB(x^k))$ is
$\rho_k^{q_k}(\lambda_i)$, where $\lambda_i$'s are the eigenvalues of $\cB(x^k)$.
According to Lemma \ref{lem:cheb-est}, the eigenvalues have the following distribution after we apply
the Chebyshev polynomial:
\[
    \min_{i \in \cI_k} \rho_k^{q_k}(\lambda_i) \geq 2^{q_kd_k\min\{\sqrt{l},1\}-q_k} \text{ and } \max_{i \notin \cI_k} \rho_k^{q_k}(\lambda_i) \leq 1. 
\]
It follows from \cite[Theorem 8.2.2]{golub2012matrix} that
\be\label{eq:err-oi}
    \|\sin\Theta(V_{\cI_k}^{k}, U^{k})\|_2 \leq \frac{2^{q_k}}{2^{q_kd_k\min\{l,1\}}} \frac{\gamma}{\sqrt{1-\gamma^2}}. 
    \ee
Due to the boundness of $\cA$ and the Lipschitz continuity of $\Phi$, we have
\be \label{eqn:ek}
\|e^k\|_2 \leq c_1\|\cP_{U^k}(\cB(x^k)) - \cP_{V_{\cI_k}^{k}}(\cB(x^k))\|_F,
\ee
where $c_1$ is a constant.
It is shown in \cite[Theorem 2.5.1]{golub2012matrix} that
$
\|\sin \Theta(X,Y)\|_2 = \|XX^T - YY^T\|_2.
$
Using this identity, we have
\begin{equation} \label{eqn:err-proj}
\begin{split}
   & \|\cP_{U^k}(\cB(x^k)) - \cP_{V_{\cI_k}^{k}}(\cB(x^k))\|_F \\
   \leq&  \|\cP_{U^k}(\cB(x^k)) - U^k(U^k)^T(\cB(x^k))V_{\cI_k}^k(V_{\cI_k}^k)^T \|_F 
   + \|U^k(U^k)^T(\cB(x^k))V_{\cI_k}^k(V_{\cI_k}^k)^T - \cP_{V_{\cI_k}^{k}}(\cB(x^k))\|_F \\
   \leq& 2||\cB(x^k)||_F\|\sin\Theta(V_{\cI_k}^{k}, U^{k})\|_2 \\
   \leq& \frac{c_2\cdot 2^{q_k}}{2^{q_kd_k\min\{l,1\}}},
\end{split}
\end{equation}
where $c_2$ is a constant depending on $\gamma$. The second inequality is due to
the fact that $\|AB\|_F \leq \|A\|_2\|B\|_F$ and the last inequality follows
from \eqref{eq:err-oi} and the boundedness of $\cA$ and $x^k$.
It is easy to verify that $\|e^k\|\leq \varepsilon$ if \eqref{eq:poly-times} holds for any $\varepsilon \in (0, 1)$.
This completes the proof.
\end{proof}

We claim that $\nabla F(x)$ is Lipschitz continuous due to the boundness of $\cA$ and $x^k$ and the Lipschitz continuity of $\Phi$. In the following part, we define the Lipschitz constant of $\nabla F(x)$ by $L$. The following theorem gives the convergence of the PFPG method. 
\begin{theorem}\label{thm:conv1}
Suppose that Assumptions \ref{assump:f} and \ref{assump:gcvx} hold.
Let $ \tau_k = \tau \leq \frac{1}{L}$ and  $\bar{x}_K = \frac{1}{K}\sum_{k=1}^K x^k$.
Then the convergence $\lim_{K \rightarrow \infty} h(\bar{x}^K) = h(x^*)$ holds if
\be
d_k = \Omega\left(\frac{\log k}{\min\{l,1\}}\right).
\ee

\end{theorem}

\begin{proof}
For simplicity we define the noisy residual function by
\[
r_{\tau}(x,e) = \frac{1}{\tau}(x - \prox_{\tau R}(x-\tau (\nabla F(x)+e))).
\]
Then the update \eqref{eq:pfpg1} can be written as
\[
x^{k} - x^{k+1} = \tau_k r_{\tau_k}(x^k, e^k).
\]
According to the Lipschitz continuity of $\nabla F$, we have
\[
    h(x^{k+1}) \leq F(x^k) - \tau_k\nabla F(x^k)^T(r_{\tau_k}(x^k,e^k)) + \frac{L \tau_k^2}{2}\|r_{\tau_k}(x^k,e^k)\|_2^2 + R(x^{k+1}). \\
\]
By the convexity of $F(x)$ and $R(x)$ and $r_{\tau_k}(x,e^k)-\nabla F(x^k)-e^k \in \partial R(x^{k+1})$, we have
\beaa
    h(x^{k+1}) &\leq& F(x^*) + \nabla F(x^k)^T(x^k-x^*)  - \tau_k\nabla F(x^k)^T(r_{\tau_k}(x^k,e^k)) + \frac{L \tau_k^2}{2}\|r_{\tau_k}(x^k,e^k)\|_2^2 \\
    &&   + R(x^*) + (r_{\tau_k}(x^k,e^k)-\nabla F(x^k)-e^k)^T(x^{k+1}-x^*). 
\eeaa
Setting the step size $\tau_k = \tau \leq \frac{1}{L}$ yields 
\beaa
    h(x^{k+1}) - h(x^*) &\leq&  \nabla F(x^k)^T(x^k-x^*)  - \tau\nabla F(x^k)^T(r_{\tau}(x^k,e^k))  
  +  \frac{\tau}{2}\|r_{\tau}(x^k,e^k)\|_2^2 \\
  &&  + (r_{\tau}(x^k,e^k)-\nabla F(x^k)-e^k)^T(x^{k+1}-x^*).
\eeaa
%\beaa
%    h(x^{k+1}) - h(x^*) &\leq&  \nabla F(x^k)^T(x^k-x^*)  - \tau\nabla F(x^k)^T(r_{\tau}(x,\epsilon_k))  
%  +  \frac{\tau}{2}\|r_{\tau_k}(x,\epsilon_k)\|_2^2 \\
%  &&  + (r_{\tau}(x,\epsilon_k)-\nabla F(x^k)-\epsilon_k)^T(x^{k+1}-x^*) \\
%  &=& r_{\tau}(x,\epsilon_k)^T(x^k-x^*) - \frac{\tau}{2}\|r_{\tau_k}(x,\epsilon_k)\|_2^2 -\epsilon_k^T(x^{k+1}-x^*) \\
%  &\leq&\frac{1}{2\tau}(\|x^{k}-x^*\|_2^2 - \|x^{k+1}-x^*\|) + C\|\epsilon_k\|,
%\eeaa
After rearranging and applying the boundness of the sequence $x^k$, we have 
\beaa
    h(x^{k+1}) - h(x^*) &\leq& \frac{1}{2\tau}(\|x^{k}-x^*\|_2^2 - \|x^{k+1}-x^*\|_2^2) + (C + \|x^*\|)\|e^k\|_2,
\eeaa
where $C$ is a constant. Then it follows that
\begin{equation} \label{eqn:mean-conv}
h(\bar{x}^K) - h(x^*) \leq \frac{1}{K}\sum_{k=1}^K (h(x^{k}) - h(x^*)) \leq 
\frac{1}{2 K\tau}\|x^0 - x^*\|_2^2 + \frac{C+\|x^*\|}{K}\sum_{k=1}^K\|e^k\|_2.
\end{equation}
According to Lemma \ref{lem:ctr1}, by setting $d_k=\Omega\left(\frac{\log k}{\min\{l,1\}}\right)$ at iteration $k$, we have the upper bound $\sum_{k=1}^K\|e^k\|_2 \leq \sum_{k=1}^{K}1/k$. Thus the right hand side of \eqref{eqn:mean-conv} goes to zero as
$K\rightarrow \infty$, which completes the proof.
\end{proof}

The following theorem gives a linear convergence rate under the strongly convex
assumption of $F(x)$.
\begin{theorem}
    Suppose that Assumptions \ref{assump:f} and \ref{assump:gcvx} hold, and $F(x)$ is $\mu$-strongly convex.
    Let $ \tau_k = \tau \leq \frac{1}{L}$ and $w\leq \frac{\mu}{2}$. Then a linear
    convergence rate is achieved if 
	\begin{equation}
	d_k = \Omega\left(\frac{-\log (w\|x^{k+1}-x^*\|_2)}{\min\{l,1\}}\right).
	\end{equation}
\end{theorem}

\begin{proof}
    The quadratic bound in the proof of Theorem \ref{thm:conv1} gives
\[
    h(x^{k+1}) \leq F(x^k) - \tau_k\nabla F(x^k)^T(r_{\tau_k}(x^k,e^k)) + \frac{L \tau_k^2}{2}\|r_{\tau_k}(x^k,e^k)\|_2^2 + R(x^{k+1}). \\
\]
By the stongly convexity of $F(x)$, the convexity of $R(x)$ and $r_{\tau_k}(x^k,e^k)-\nabla F(x^k)-e^k \in \partial R(x^{k+1})$, we have
\beaa
h(x^{k+1}) &\leq& F(x^*) + \nabla F(x^k)^T(x^k-x^*) - \frac{\mu}{2}\|x^k-x^*\|_2^2 - \tau_k\nabla F(x^k)^T(r_{\tau_k}(x^k,e^k)) \\
&& + \frac{L \tau_k^2}{2}\|r_{\tau_k}(x^k,e^k)\|_2^2 
       + R(x^*) + (r_{\tau_k}(x^k,e^k)-\nabla F(x^k)-e^k)^T(x^{k+1}-x^*). 
\eeaa
Under the choice of step size $\tau_k = \tau \leq \frac{1}{L}$, we have 
\beaa
    h(x^{k+1}) - h(x^*) &\leq&  \nabla F(x^k)^T(x^k-x^*)  - \tau\nabla F(x^k)^T(r_{\tau}(x^k,e^k))  
  +  \frac{\tau}{2}\|r_{\tau}(x^k,e^k)\|_2^2 \\
  &&- \frac{\mu}{2}\|x^k-x^*\|_2^2   + (r_{\tau}(x^k,e^k)-\nabla F(x^k)-e^k)^T(x^{k+1}-x^*) \\
  &=& \frac{1}{2\tau}(( 1-\mu \tau) \|x^{k}-x^*\|_2^2 - \|x^{k+1}-x^*\|_2^2) - (e^k)^T(x^{k+1}-x^*) \\
  &\leq& \frac{1}{2\tau}(( 1-\mu \tau) \|x^{k}-x^*\|_2^2 - \|x^{k+1}-x^*\|_2^2) + \|e^k\|_2\|x^{k+1}-x^*\|_2.
  \eeaa
Together with $h(x^{k+1}) \geq h(x^*)$, we have
\begin{equation} \label{eqn:recurr-lin}
    \|x^{k+1}-x^*\|_2^2 \leq ( 1-\mu \tau) \|x^{k}-x^*\|_2^2 + 2\tau\|e^k\|_2\|x^{k+1}-x^*\|_2.
\end{equation}
The choice of $d_k$ guarantees
$
\|e^k\|_2 \leq w\|x^{k+1}-x^*\|_2.
$
Plugging this upper bound of $\|e^k\|$ into \eqref{eqn:recurr-lin} yields 
\begin{equation*}
\|x^{k+1} - x^*\|^2_2 \leq \eta \|x^k - x^*\|_2^2,
\end{equation*}
where $\eta =\frac{ 1-\mu \tau}{1 - 2w \tau} < 1 $, which completes the proof.
\end{proof}

%\subsection{Gap-independent convergence analysis}
\subsection{Warm-start analysis for PFPG}
In this section, we analyze the linear convergence of PFPG under the warm-start setting. 
Suppose that $\rho_k(\lambda_{s_i}(\cB(x^{k}))$ are in decreasing order, and define 
\begin{equation} \label{eqn:etak}
	\eta_k = \frac{\rho_k(\lambda_{s_{p+1}}(\cB(x^{k}))}{\rho_k(\lambda_{s_p}(\cB(x^{k}))}.
\end{equation}
We should point out that under Assumption \ref{assump:g}, $\eta_k$ is strictly
smaller than 1 since
$\cG_k \geq l > 0$.
An additional assumption is needed if we intend to refine the subspace from the previous one.
\begin{assumption}\label{assump:g}
    $\|\sin\Theta(V^{k+1}_{\cI_{k+1}}, V^{k}_{\cI_{k}})\|_2 \leq c_1\|x^{k+1} - x^{k}\|_2$ for all $k$.
\end{assumption}
Assumption \ref{assump:g} means that the eigen-space at two consecutive iteration points are close enough. This is necessary since we use the eigen-space at the previous step as the initial value of polynomial filter. This assumption is satisfied when the gap is large enough by the Davis-Kahan $\sin \Theta$ theorem \cite{davis1970rotation}. 

%
%According to the Corollary 8.1.11 in \cite{golub2012matrix}, we have the following lemma.
%\begin{lemma}
%    Suppose Assumption \ref{assump:g} holds. Then we have
%    \be
%    \|\sin\Theta(\bar{U}^{k+1}, \bar{U}^k)\|_2 \leq \frac{c_4\tau_k}{1-\eta_{k+1}}\|x^k - x^{k+1}\|_2 
%    \ee
%\end{lemma}
%\begin{proof}
%    (TODO). Maybe we need more assumption.
%\end{proof}
%

The next lemma shows the relationship between the two iterations in a recursive form.
It plays a key role in the proof of the main convergence theorem.
\begin{lemma}\label{lem:recur}
    Suppose that Assumptions \ref{assump:f}, \ref{assump:gcvx} and \ref{assump:g} hold. 
    Let $\cX$ be the set of optimal solutions of the problem \eqref{eqn:mat-opt}, 
    and $\mathrm{dist}(x, \cX) = \inf_{y\in \cX}\|x - y\|_2$ be the distance between $x$
    and $\cX$. Then we have
\be \label{eqn:recr1}
\|\sin\Theta(V^{k+1}_{\cI_{k+1}}, U^{k+1})\|_2 \leq 
\eta_{k+1}^{q_k}(c_2\|\sin\Theta(V^{k}_{\cI_{k}}, U^k)\|_2 + c_3\dist(x^k,\cX)), 
\ee
%where $c_2= \frac{1+2c_1L_{\cA}L_{\Phi}(\|G\|_2+L_{\cA}M)}{\sqrt{1-\gamma^2}}$ and $c_3 = \frac{c_1 L }{\sqrt{1-\gamma^2}}$ 
where $c_2$ and $c_3$ are constants. 
\end{lemma}
\begin{proof}
The update \eqref{eqn:subspace-update} can be seen as one iteration in the orthogonal iteration. Then it follows from \cite[Theorem 8.2.2]{golub2012matrix} that
\bee
\|\sin\Theta(V^{k+1}_{\cI_{k+1}}, U^{k+1})\|_2 \leq \eta_{k+1}^{q_k} \|\tan\Theta(V^{k+1}_{\cI_{k+1}}, U^k)\|_2 \leq \frac{\eta_{k+1}^{q_k}}{\sqrt{1-\gamma^2}} \|\sin\Theta(V^{k+1}_{\cI_{k+1}}, U^k)\|_2, 
\eee
According to \eqref{eqn:ek} and \eqref{eqn:err-proj} in Lemma \ref{lem:ctr1}, we have
\beaa
\|e^k\|_2 \leq c_4\|\sin\Theta(V^{k}_{\cI_{k}}, U^k)\|_2,
\eeaa
where $c_4$ is a constant. 
The upper bound of $\|x^k - x^{k+1}\|_2$ is given by
\beaa
&&\|x^k - x^{k+1}\|_2 \leq \tau_k(\|r_{\tau_k}(x^k,e^k) - r_{\tau_k}(x^k,0)\|_2 + \|r_{\tau_k}(x^k,0)\|_2)  \\
&\leq& \tau_k\|e^k\|_2 + \omega \dist(x^k, \cX) \\
&\leq& \tau_kc_4\|\sin\Theta(V_{\cI_{k}}^{k}, U^{k})\|_2 + \omega\dist(x^k,\cX),
\eeaa
where the second inequality is due to the monotonity of the proximal operator and the cocoercivity with modulus $\omega$ of the residual function $r_{\tau_k}(x^k, 0)$.
Finally, we obtain
\beaa
&&\|\sin\Theta(V^{k+1}_{\cI_{k+1}}, U^{k+1})\|_2 \leq \frac{\eta_{k+1}^{q_k}}{\sqrt{1-\gamma^2}} \|\sin\Theta(V^{k+1}_{\cI_{k+1}}, U^k)\|_2   
 \\
 &\leq& \frac{\eta_{k+1}^{q_k}}{\sqrt{1-\gamma^2}} (\|\sin\Theta(V^{k}_{\cI_{k}}, U^k)\|_2+\|\sin\Theta(V^{k+1}_{\cI_{k+1}}, V^{k}_{\cI_{k}}\|_2) \\
&\leq&  \frac{\eta_{k+1}^{q_k}}{\sqrt{1-\gamma^2}}(\|\sin\Theta(V^{k}_{\cI_{k}}, U^k)\|_2 + c_1\|x^k - x^{k+1}\|_2) \\
&\leq& \frac{\eta_{k+1}^{q_k}}{\sqrt{1-\gamma^2}} ((1+\tau_kc_1c_4)\|\sin\Theta(V^{k}_{\cI_{k}}, U^k)\|_2 + c_1\omega\dist(x^k,\cX)), \\
\eeaa
where the third inequality is from (iii) in Assumption \ref{assump:g}. This completes the proof. 
\end{proof}

Lemma \ref{lem:recur} means that the principle angle between $V^{k+1}_{\cI_{k+1}}$ and $U^{k+1}$ are
determined by the angle in the previous iteration and how far $x^k$ is from the optimal set $\cX$.
As an intuitive interpretation, the polynomial-filtered subspace becomes
more and more accurate when PFPG is close to converge.

We are now ready to state the main result. It implies that the PFPG method has
a linear convergence rate if the polynomial degree is a constant, independent
of the iteration $k$.
\begin{theorem}\label{thm:lr}
    Suppose that Assumptions \ref{assump:f}, \ref{assump:gcvx} and \ref{assump:g} holds, and the exact proximal 
    gradient method for \eqref{eqn:mat-opt} has a linear convergence rate, i.e.,
	\be
    \dist(\mathrm{prox}_{\tau R}(x^{k}- \tau_k \nabla F(x^k)),\cX) \leq	\nu\dist(x^{k},\cX), \nu \in (0,1). 
	\ee
    If $\eta_{k+1}$ \eqref{eqn:etak} satisfies  
    \be\label{eq:cond}
    \frac{\nu + \eta_{k+1}^{q_k}c_3}{2} + \sqrt{\left(\frac{\nu+\eta_{k+1}^{q_k} c_3}{2}\right)^2+\eta_{k+1}^{q_k}(\tau_k c_2 c_4 -\nu c_3)} < \rho < 1,
    \ee
    where $c_2,c_3,c_4$ are constants in Lemma \ref{lem:recur},  
	then the PFPG method has a linear convergence rate.
\end{theorem}
\begin{proof}
    Since $\cX$ is closed, we denote $x_{prj}$ as the projection of $\mathrm{prox}_{\tau_k R}(x^k-\tau_k\nabla F(x^k))$ onto the optimal set $\cX$. Then we have
\begin{equation}
\begin{split}\label{eqn:recr2}
&\dist(x^{k+1},\cX) 
\leq \|\prox_{\tau_k R}(x^{k}- \tau_k (\nabla F(x^k)+e^k)) -x_{prj}\|_2  \\
\leq&    \|\prox_{\tau_k R}(x^{k}- \tau_k\nabla F(x^k))  - \prox_{\tau_k R}(x^{k}- \tau_k(\nabla F(x^k) + e^k)) \|_2 \\
& + \|\prox_{\tau_k R}(x^{k}- \tau_k\nabla F(x^k)) - x_{prj}\|_2 \\
\leq&  \dist(\prox_{\tau_k R}(x^{k}- \tau_k \nabla F(x^k)),\cX) + \tau_k \|e^k\|_2  \\
\leq& \dist(\prox_{\tau_k R}(x^{k}- \tau_k \nabla F(x^k)),\cX) + \tau_k c_4\|\sin\Theta(V^{k}_{\cI_{k}}, U^k)\|_F \\
\leq& \nu\dist(x^{k},\cX) + \tau_kc_4\|\sin\Theta(V_{\cI_{k}}^{k}, U^{k})\|_F,\
\end{split}
\end{equation}
where the first inequality is from the definition of the distance function, the second inequality is from the triangle inequality, the third inequality is shown in the proof of Lemma \ref{lem:recur} and the last inequality is due to the linear convergence of the proximal gradient method. 

From \eqref{eqn:recr1} and \eqref{eqn:recr2} we observe that the error $\dist(x^{k},\cX)$
and $\|\sin\Theta(V_{\cI_{k}}^{k}, U^{k})\|_2$ are coupled together. Thus we 
define the error vector 
\begin{equation}
s^k = \left[\begin{matrix}
\dist(x^{k},\cX) \\
\|\sin\Theta(V_{\cI_{k}}^{k}, U^{k})\|_2
\end{matrix}\right].
\end{equation}
The two recursive formula can be rewritten as
\be
s^{k+1} \leq R^k s^k,
%where  
\
R^k =
\left[
\begin{matrix}
    \nu & \tau_kc_4 \\
    \eta_{k+1}^{q_k} c_2 & \eta_{k+1}^{q_k} c_3 \\
\end{matrix} \right].
\ee
The spectral radius of $R^k$ is bounded by the left term in \eqref{eq:cond}. It
is easy to verify that \eqref{eq:cond} holds if the degree $d_k$ is large enough
to make $\eta_{k+1}$ sufficiently small. The choice of $d_k$ is independent of
the iteration number $k$. This completes the proof.
\end{proof}

Theorem \ref{thm:lr} requires the linear convergence rate of the exact proximal gradient
method. This condition is satisfied if $F(x)$ is strongly convex and smooth.
From the proof we also observe that $\dist(x^{k},\cX)$ and $\|\sin\Theta(V_{\cI_{k}}^{k}, U^{k})\|_2$
are recursively bounded by each other. By choosing suitable $d_k$, the two errors
are able to decay simultaneously.

\subsection{Convergence analysis for PFAM}
We next analyze the convergence of PFAM. 
Similar to PFPG, we make the following assumptions:
\begin{assumption}\label{assump:PFAM}
	Let $V_{\cI_{k}}^k$ be an orthogonal matrix whose columns are the eigenvectors corresponding to
	all positive eigenvalues of $Z^k$.
    \begin{enumerate}[(i)]
        \item $\|\sin \Theta(V^{k+1}_{\cI_{k+1}},U^k)\|_2 < \gamma, \ \forall \ k$ with $\gamma < 1$.
        %\item The noisy gradients are bounded, i.e., $\|\mathcal{A}(\Phi(U^k (U^k)^T(\cB(x^k))U^k (U^k)^T)\| \leq M_1$. 
        \item The sequence generated by  \eqref{eq:PFAM1} and \eqref{eq:PFAM2} is bounded, i.e.,
        $
        \|Z^k\|_F \leq C, \ \forall \ k. 
        $
        \item The relative gap \eqref{eqn:relative-gap} has a lower bound, i.e., $\cG_k \geq l, \ \forall \ k$. 
    \end{enumerate}
\end{assumption}

The main theorem is established by following the proof in \cite{davis2016convergence}. 
\begin{theorem}\label{thm:PFAM}
Suppose Assumption \ref{assump:PFAM} holds.
The convergence $\|Z^k - T_{\textrm{DRS}}(Z^k)\|_F = \cO(1/\sqrt{k})$ is archived
if
\be
d_k = \Omega\left(\frac{\log k}{\min\{l,1\}}\right).
\ee
\end{theorem}
\begin{proof}
  According to the proof of Lemma \ref{lem:ctr1}, we obtain
\beaa
   && \|\cP_{U^k}(Z^k) - \cP_{V_{\cI_k}^{k}}(Z^k)\|_F 
   \leq \frac{c\cdot 2^{q_k}}{2^{q_kd_k\min\{l,1\}}},
\eeaa
where $c$ is a constant.
Define the error between one PFAM update and one DRS iteration $E^k = Z^{k+1}-T_{\textrm{DRS}}(Z^k)$.
Since $\prox_{tG}$ and $\cP_+$ are Lipschitz 
continuous with constant $L_1$ and $L_2$, we have
\[
    \|E^k\|_F \leq (2L_1L_2 + L_2)\|\cP_{U^k}(Z^k) - \cP_{V_{\cI_k}^{k}}(Z^k)\|_F \leq  \frac{c_2\cdot 2^{q_k}}{2^{q_kd_k\min\{l,1\}}},
\]
where $c_2$ is a constant. Let $s>1$ be an arbitrary number, the error can be controlled with an order
\[
\|E^k\|_F = \cO(\frac{1}{k^s}),
\]
by choosing $d_k = \Omega(\frac{s\log l}{\min(l, 1)})$.

Define $W^k =T_{\textrm{DRS}}(Z^k) - Z^k$, $S^k = Z^{k+1}-Z^k$ 
and $T^k = T_{\textrm{DRS}}(Z^{k+1}) - T_{\textrm{DRS}}(Z^k)$.
Then we have
\begin{equation} \label{eqn:Wk1}
\begin{split}
\|W^{k+1}\|_F^2 &= \|W^k\|_F^2 + \|W^{k+1} - W^{k}\|_F^2 + 2\iprod{W^k}{W^{k+1}-W^k}\\
&= \|W^k\|_F^2 + \|T^{k}-S^k\|_F^2 + 2\iprod{S^k-E^k}{T^{k}-S^k}.
\end{split}
\end{equation}
By the firm nonexpansiveness of $T_{\text{DRS}}$, we obtain
\begin{equation} \label{eqn:Wk2}
 2\iprod{S^k}{T^{k}-S^k} = \|T^k\|^2_F -\|S^k\|_F^2 - \|T^k-S^k\|_F^2
 \leq -2\|T^k-S^k\|_F^2.
\end{equation}
Plugging \eqref{eqn:Wk2} into \eqref{eqn:Wk1} yields
\be\label{eq:bound1}
\|W^{k+1}\|_F^2 \leq \|W^k\|_F^2 - \|T^k-S^k\|_F^2 - 2\iprod{E^k}{T^{k}-S^k}
\leq \|W^k\|_F^2 + \|E^k\|_F^2.
\ee
Since $\sum_{k=0}^{\infty}\|E^k\|_F^2 < \infty$, it implies the boundness of
$\|W^k\|_F^2$, i.e., 
\[
    \|W^k\|_F^2 \leq \|W^0\|_F^2 + \sum_{i=0}^{k-1}\|E^k\|_F^2 < \infty.
\]
Let $Z^*$ be a fixed point of $T_{\textrm{DRS}}$, then we obtain that
\begin{equation} \label{eqn:TDRS1}
    \|T_{\textrm{DRS}}(Z^k) - Z^*\|_F \leq \|Z^k-Z^*\|_F 
    \leq \|T_{\textrm{DRS}}(Z^{k-1})-Z^*\|_F + \|E^{k-1}\|_F.
\end{equation}
Applying \eqref{eqn:TDRS1} recursively gives
\[
    \|T_{\textrm{DRS}}(Z^k) - Z^*\|_F \leq \|T_{\textrm{DRS}}(Z^{0})-Z^*\|_F +\sum_{i=0}^{k-1}\|E^k\|_F < \infty.
\]
Hence $\|T_{\textrm{DRS}}(Z^k) - Z^*\|_F$ is bounded.
Due to the firm nonexpansiveness of $T_{\text{DRS}}$, we have that
\beaa
\|Z^{k+1} - Z^*\|_F^2 &=& \|T_{\textrm{DRS}}(Z^k) - Z^*\|_F^2 + \|E^k\|_F^2 
+ 2\iprod{T_{\textrm{DRS}}(Z^k) - Z^*}{E^k} \\
&\leq& \|Z^k - Z^*\|_F^2 - \|W^k\|_F^2 + 2\|T_{\textrm{DRS}}(Z^k) - Z^*\|_F\|E^k\|_F.
\eeaa
Therefore, it holds 
\be\label{eq:bound2}
\sum_{i=0}^{\infty}\|W^i\|_F^2 \leq \|Z^0-Z^*\|_F^2 + 2 \sum_{i=0}^{\infty}\|T_{\textrm{DRS}}(Z^i) - Z^*\|_F\|E^i\|_F \leq \infty.
\ee
The upper bound \eqref{eq:bound2} implies $\|W^k\|\rightarrow 0$. In fact, we can also
show the convergence speed.
According to \eqref{eq:bound1} and \eqref{eq:bound2}, we have
\beaa
\left\lceil\frac{k}{2}\right\rceil\|W^k\|_F^2
&=& \sum_{i = \left\lceil\frac{k}{2}\right\rceil }^k\|W^k\|_F^2 
\leq \sum_{i = \left\lceil\frac{k}{2}\right\rceil }^k 
(\|W^i\|_F^2 + \sum_{j=i}^k\|E^j\|_F^2) \\
&=& \sum_{i = \left\lceil\frac{k}{2}\right\rceil }^k \|W^i\|_F^2 
+ \sum_{i = \left\lceil\frac{k}{2}\right\rceil }^k (k-i+1)\|E^i\|_F^2 \\
&\leq& \sum_{i = \left\lceil\frac{k}{2}\right\rceil }^k \|W^i\|_F^2 
+ c\sum_{i = \left\lceil\frac{k}{2}\right\rceil }^k \|E^i\|_F \longrightarrow 0,
\eeaa
where the last inequality is due to the fact that $\|E^k\|_F$ dominates $\|E^k\|_F^2$ since $\|E^k\|_F = \cO(1/k^s)$. This completes the proof.
\end{proof}

\section{Implementation Details} \label{sec:impls}
%\subsection{Evaluating the Chebyshev Filtering}
%The computation of $U\leftarrow \rho(A)U$ cannot be evaluated directly by \eqref{eq:cheb}.
%In practice, the three-term recursion definition is utilized to evaluate
%a Chebyshev filter.
%After a linear map of the suppressing interval defined by \eqref{eqn:cheb-linear-map},
%the recursion formula becomes
%\begin{equation} \label{eqn:cheby-recursion-map}
%	\rho_{d+1}(t) = 2(c_0 + c_1 t)\rho_d(t) - \rho_{d-1}(t),
%\end{equation}
%where $c_0 = 2 / (b - a), c_1 = (a + b) / (a - b)$.

\subsection{Choice of the Interval $[a, b]$}
We mainly focus on how to
extract a good subspace that contains the eigenvectors corresponding to 1) all positive eigenvalues
and 2) $r$ largest eigenvalues, by choosing proper $[a, b]$.
In either case, a good choice of $a$ is the smallest eigenvalue of
the input matrix $A$, say $\lambda_n$.
For example, one can use the Lanczos method (\texttt{eigs} in \textsc{matlab})
to estimate $\lambda_n$. Note that $A$ is  changed during the iterations, thus
$a$ must be kept up to date periodically.

The estimation of $b$ is related to the case we are dealing with. If all positive eigenvalues are computed, $b$ is set to $b = \eta a$ where $\eta$ is a
small positive number, usually 0.1 to 0.2.
If $r$ largest eigenvalues are needed only, 
then we choose $b=(1 - \eta)\hat{\lambda}_r + \eta a$, where $\hat{\lambda}_r$ is one estimation
of $\lambda_r$. A natural choice is the $\lambda_r$ in the previous iteration.

\subsection{Choice of the Subspace Dimension}
We next discuss about the choice of $p$, the number of columns of $U^k$. Suppose $r$ largest
eigenvalues of a matrix is required, we can add some guard vectors to $U^k$ and set $p=5+q\cdot r$ where $q>1$
is a parameter. If we want to compute all positive eigen-pairs, the number of positive eigenvalues
in the previous iteration is used as an estimation since the real dimension is unknown. The guard vectors are also added in order to preserve
the convergence in case the dimension is underestimated.

\subsection{Acceleration Techniques}
Gradient descent algorithms usually have slow convergence. This
issue can be resolved using acceleration techniques such as
Anderson Acceleration (AA) \cite{anderson1965iterative,walker2011anderson}.
In our implementation we apply an extrapolation-based acceleration
techniques proposed in \cite{scieur2016regularized} to overcome the
instability of the Anderson Acceleration. To be precise,
we perform linear combinations of the points $x^k$ every
$l+2$ iterations to obtain a better estimation $\tilde{x} = \sum_{i=0}^l \tilde{c}_i x^{k-l+i}$.
Define the difference of $l+2$ iteration points
\[
U=[x^{k-l+1} - x^{k-l}, \ldots, x^{k+1} - x^k],
\]
the coefficients $\tilde{c} = (\tilde{c}_0,\ldots,\tilde{c}_l)^T$ is the solution
of the following problem
\begin{equation}
	\tilde{c} = \argmin_{\mathbf{1}^T c = 1} c^T(U^TU + \lambda I)c,
\end{equation}
where $\lambda > 0$ is a regularization parameter.

\section{Numerical Experiments} \label{sec:num}
This section reports on a set of applications and their numerical
results based on the spectral optimization problem \eqref{eqn:mat-opt}.
All experiments are performed on a Linux server with two 
twelve-core Intel Xeon E5-2680 v3 processors at 2.5 GHz
and a total amount of 128 GB shared memory.
All reported time is wall-clock time in seconds.
\subsection{Nearest Correlation Estimation}
Given a matrix $G \in \mathcal{S}^n$, the (unweighted) nearest correlation estimation
problem (NCE) is to solve the semi-definite optimization problem:
\begin{equation}\label{eqn:NCE-primal}
\begin{array}{rl}
\min & \frac{1}{2}\|G-X\|_F^2, \\
\mathrm{s.t.} & \mathrm{diag}(X) = \mathbf{1}_n, \\
& X \succeq 0.
\end{array}
\end{equation} 
The dual problem of \eqref{eqn:NCE-primal} is
\begin{equation} \label{eqn:NCE-dual}
\min  \frac{1}{2}\|\Pi_{\mathcal{S}^n_+}(G+\mathrm{Diag}(x))\|_F^2 -\mathbf{1}_n^Tx,
\end{equation}
where $\Pi_{\mathcal{S}_{+}^n}(\cdot)$ is the projection operator onto the positive semi-definite
cone. Note that \eqref{eqn:NCE-dual} can be rewritten as a
spectral function minimization problem as \eqref{eqn:mat-opt},
whose objective function and gradient are:
\begin{eqnarray} 
	\min F(x)&:=&f\circ \lambda(G+\mathrm{Diag}(x)) - \mathbf{1}_n^Tx, \label{eqn:NCE-dual-equiv} \\
	\nabla F(x) &=& \mathrm{diag}(\Pi_{\mathcal{S}^n_+}(G + \mathrm{Diag}(x))) - \mathbf{1}_n,\label{eqn:NCE-gradient}
\end{eqnarray}
where $f(\lambda)=\frac{1}{2}\sum_{i=1}^n\max(\lambda_i, 0)^2$.
Expression \eqref{eqn:NCE-dual-equiv} and \eqref{eqn:NCE-gradient} imply
only positive eigenvalues and eigenvectors of $G+\mathrm{Diag}(x^k)$ are needed
at each iteration. Hence Assumption \ref{assump:f} is satisfied if
$\Pi_{\mathcal{S}^n_+}(G + \mathrm{Diag}(x))$ is low-rank for all $x$
in the neighborhood of $x^*$. 

First we generate synthetic NCE problem data based on the
first three examples in \cite{doi:10.1137/050624509}.
\begin{example} \label{eg:5.5}
	$C \in \cS^{n}$ is randomly generated correlation matrix
	using \texttt{gallery('randcorr', n)} in \textsc{matlab}.
	$R$ is an $n$ by $n$ random matrix whose entries $R_{ij}$
	are sampled from the uniform distribution in $[-1, 1]$. Then the
	matrix $G$ is set to $G = C + R$.
\end{example}
\begin{example} \label{eg:5.6}
	$G \in \cS^n$ is a randomly generated  matrix with $G_{ij}$
	satisfying the uniform distribution in $[-1, 1]$ if $i\neq j$.
	The diagonal elements $G_{ii}$ is set to 1 for all $i$.
\end{example}
\begin{example} \label{eg:5.7}
	$G \in \cS^n$ is a randomly generated  matrix with $G_{ij}$
	satisfying the uniform distribution in $[0, 2]$ if $i\neq j$.
	The diagonal elements $G_{ii}$ is set to 1 for all $i$.
\end{example}
In all three examples, the dimension $n$ is set to
various numbers from 500 to 4000. The initial point $x^0$ is set
to $\mathbf{1}_n - \mathrm{diag}(G)$. Three methods are compared in
NCE problems, namely, the gradient method with a fixed step size (Grad)
and our proposed method (PFPG), and
the semi-smooth Newton method (Newton) proposed in \cite{doi:10.1137/050624509}.
The stopping criteria is $\|\nabla F(x)\|\leq 10^{-7}\|\nabla F(x_0)\|$.
The reason not to use a higher accuracy is that first-order methods (such as Grad)
are not designed for obtaining high accuracy solutions.
We also record the number of iterations (denoted by ``iter''), the
objective function value $f$ and the relative norm of the gradient
$\|g\| = \|\nabla F(x)\|/\|\nabla F(x^0)\|$. The results are shown in Table \ref{tab:ncm}.

%\begin{table}[H] \caption{Result of Example \ref{eg:5.5}}
%\centering
%\begin{tabular}{|c|ccccc|ccccc|} \hline
%	\input{tab/Test_5_5.tex}
%	
%\end{tabular}
%\end{table}
%
%\begin{table}[H] \caption{Result of Example \ref{eg:5.6}}
%	\centering
%	\begin{tabular}{|c|ccccc|ccccc|} \hline
%		\input{tab/Test_5_6_1.tex}
%		
%	\end{tabular}
%\end{table}
%
%\begin{table}[H] \caption{Result of Example \ref{eg:5.7}}
%	\centering
%	\begin{tabular}{|c|ccccc|ccccc|} \hline
%		\input{tab/Test_5_7_1.tex}
%		
%	\end{tabular}
%\end{table}
\begin{table}[h]
	\centering
	\setlength{\tabcolsep}{3pt}
	\caption{Results of NCE problems} \label{tab:ncm}
	\begin{tabular}{|c|cccc|cccc|cccc|} \hline
		\multicolumn{13}{|c|}{Results of Example \ref{eg:5.5}} \\ \hline
		\multirow{2}{*}{$n$} & \multicolumn{4}{c|}{Grad} &\multicolumn{4}{c|}{PFPG} & \multicolumn{4}{c|}{Newton} \\ \cline{2-13} 
		& time & iter & $f$ & $\|g\|$ & time & iter & $f$ & $\|g\|$ & time & iter & $f$ & $\|g\|$ \\ \hline 
		500 &   0.8 &   28 & 1.1e+05 & 7.9e-08 &   0.7 &   40 & 1.1e+05 & 4.1e-08 &   0.6 &    8 & 1.1e+05 & 1.7e-08 \\ 
		1000 &   2.6 &   28 & 3.2e+05 & 5.5e-08 &   1.1 &   46 & 3.2e+05 & 2.6e-08 &   1.4 &    8 & 3.2e+05 & 1.2e-08 \\ 
		1500 &   6.2 &   28 & 5.9e+05 & 7.0e-08 &   2.5 &   46 & 5.9e+05 & 8.2e-08 &   3.3 &    8 & 5.9e+05 & 1.0e-08 \\ 
		2000 &  13.5 &   31 & 9.2e+05 & 8.0e-08 &   4.2 &   46 & 9.2e+05 & 4.9e-08 &   5.9 &    8 & 9.2e+05 & 8.8e-09 \\ 
		2500 &  24.9 &   34 & 1.3e+06 & 1.7e-08 &   7.2 &   47 & 1.3e+06 & 6.9e-08 &   9.6 &    8 & 1.3e+06 & 7.9e-09 \\ 
		3000 &  43.2 &   34 & 1.7e+06 & 2.8e-08 &  13.2 &   52 & 1.7e+06 & 1.8e-08 &  15.7 &    8 & 1.7e+06 & 7.2e-09 \\ 
		4000 & 110.6 &   34 & 2.7e+06 & 1.4e-08 &  25.0 &   52 & 2.7e+06 & 2.4e-08 &  36.1 &    8 & 2.7e+06 & 6.3e-09 \\ 
		\hline \hline 
		
		\multicolumn{13}{|c|}{Results of Example \ref{eg:5.6}} \\ \hline
        \multirow{2}{*}{$n$} & \multicolumn{4}{c|}{Grad} &\multicolumn{4}{c|}{PFPG} & \multicolumn{4}{c|}{Newton} \\ \cline{2-13} 
        & time & iter & $f$ & $\|g\|$ & time & iter & $f$ & $\|g\|$ & time & iter & $f$ & $\|g\|$ \\ \hline 
        500 &   0.5 &   16 & 8.9e+03 & 2.2e-09 &   0.6 &   16 & 8.9e+03 & 1.0e-08 &   0.5 &    5 & 8.9e+03 & 1.6e-07 \\ 
        1000 &   1.6 &   16 & 2.6e+04 & 7.4e-08 &   1.0 &   17 & 2.6e+04 & 8.3e-08 &   1.2 &    6 & 2.6e+04 & 1.2e-07 \\ 
        1500 &   4.0 &   17 & 5.0e+04 & 6.6e-08 &   2.2 &   20 & 5.0e+04 & 9.1e-08 &   2.5 &    6 & 5.0e+04 & 9.6e-08 \\ 
        2000 &   7.5 &   17 & 7.8e+04 & 8.9e-08 &   4.8 &   22 & 7.8e+04 & 2.4e-08 &   4.7 &    6 & 7.8e+04 & 8.3e-08 \\ 
        2500 &  13.8 &   19 & 1.1e+05 & 7.2e-08 &   7.4 &   22 & 1.1e+05 & 4.5e-08 &   7.6 &    6 & 1.1e+05 & 7.5e-08 \\ 
        3000 &  25.9 &   21 & 1.5e+05 & 8.9e-08 &  12.0 &   22 & 1.5e+05 & 8.1e-08 &  12.3 &    6 & 1.5e+05 & 6.9e-08 \\ 
        4000 &  72.5 &   22 & 2.3e+05 & 3.4e-09 &  27.2 &   25 & 2.3e+05 & 9.9e-08 &  28.6 &    6 & 2.3e+05 & 6.0e-08 \\ 
        \hline \hline 
        
		\multicolumn{13}{|c|}{Results of Example \ref{eg:5.7}} \\ \hline
        \multirow{2}{*}{$n$} & \multicolumn{4}{c|}{Grad} &\multicolumn{4}{c|}{PFPG} & \multicolumn{4}{c|}{Newton} \\ \cline{2-13} 
        & time & iter & $f$ & $\|g\|$ & time & iter & $f$ & $\|g\|$ & time & iter & $f$ & $\|g\|$ \\ \hline 
        500 &   0.9 &   33 & 1.3e+05 & 7.4e-08 &   0.7 &   43 & 1.3e+05 & 1.3e-08 &   0.6 &    8 & 1.3e+05 & 1.8e-07 \\ 
        1000 &   3.8 &   43 & 5.0e+05 & 2.0e-08 &   1.1 &   54 & 5.0e+05 & 3.0e-08 &   1.6 &    9 & 5.0e+05 & 1.3e-07 \\ 
        1500 &  11.3 &   54 & 1.1e+06 & 2.6e-08 &   2.4 &   65 & 1.1e+06 & 7.1e-08 &   4.0 &   10 & 1.1e+06 & 1.0e-07 \\ 
        2000 &  22.6 &   54 & 2.0e+06 & 4.3e-08 &   5.0 &   76 & 2.0e+06 & 8.3e-08 &   7.2 &   10 & 2.0e+06 & 9.0e-08 \\ 
        2500 &  60.9 &   87 & 3.1e+06 & 3.6e-08 &  10.6 &  120 & 3.1e+06 & 4.2e-08 &  12.1 &   10 & 3.1e+06 & 8.0e-08 \\ 
        3000 & 104.0 &   87 & 4.5e+06 & 6.2e-08 &  16.1 &  129 & 4.5e+06 & 7.8e-08 &  19.3 &   10 & 4.5e+06 & 7.4e-08 \\ 
        4000 & 278.0 &   91 & 8.0e+06 & 7.9e-08 &  32.8 &  142 & 8.0e+06 & 7.3e-08 &  44.2 &   10 & 8.0e+06 & 6.4e-08 \\ 
        \hline

	\end{tabular}
\end{table}
The following observations can be drawn from Table \ref{tab:ncm}:
\begin{itemize}
	\item All three methods reached the required accuracy in all test instances.
	For large problems such as $n=4000$, PFPG usually provides at least
	three times speedup over Grad. For some large problems such as $n=4000$ in Example \ref{eg:5.7},
	PFPG has nearly 9 times speedup.
	\item The Newton method usually converges within 10 iterations, since it enjoys
	the quadratic convergence property. The main cost of the Newton method are
	full EVDs and solving linear systems, which make each
	iteration slow. However, PFPG benefits from the subspace extraction
	so that the cost of each iteration is greatly reduced, making it very competitive
	against the Newton method.
	\item PFPG usually needs more iterations than Grad, which is resulted
	from the inexactness of the Chebyshev filters. The error
	can be controlled by their degree. Using higher order polynomials
	will reduce the number of iterations, but increase the cost per iteration.
	Moreover, the number of iterations does not increase much thanks to the
	warm start property of the filters.
\end{itemize}

It should be pointed out that the benefit brought by PFPG is highly dependent
on the number of positive eigenvalues $r_+$ at $G + \mathrm{Diag}(x^*)$. We have checked
this number and find out that for Example \ref{eg:5.5} and \ref{eg:5.6}, $r_+$
is approximately $0.1n$ , and $r_+ \approx 0.05n$ for Example \ref{eg:5.7}. That is 
the reason why PFPG is able to provide higher speedups in Example \ref{eg:5.7}.
As a conclusion we address that PFPG is not suitable for $r_+ \approx 0.5n$.

Next we test the three methods on real datasets. The test matrices are selected from
the invalid correlation matrix repository\footnote{Available at \url{https://github.com/higham/matrices-correlation-invalid}}.
Since we are more interested in large NCE problems, we choose three largest matrices
from the collection, whose details are presented in Table \ref{tab:invalid-matrix}. The label
``opt. rank'' denotes the rank of the optimal point $X^*$ in \eqref{eqn:NCE-primal}.
Note that for matrix ``cor3120'', the solution is not low-rank. However, the projection
can be performed with respect to the negative definite part of the variable, which is
low-rank in this case.

\begin{table}[h]
	\centering
	\caption{Invalid correlation matrices from real applications} \label{tab:invalid-matrix}
	\begin{tabular}{|c|c|c|l|}
		\hline
		name & size & opt. rank & description \\
		\hline
		cor1399 & $1399\times 1399$ & 154 & matrix constructed from stock data \\ \hline
		cor3120 & $3120\times 3120$ & 3025 & matrix constructed from stock data \\ \hline
		bccd16  & $3250\times 3250$ & 5 & matrix from EU bank data \\ \hline
	\end{tabular}
\end{table}

Again, we compare the performance of Grad, PFPG, and Newton on the three matrices.
The stopping criteria is set to $\|\nabla F(x)\|\leq 10^{-5}\|\nabla F(x_0)\|$.
The results are shown in Table \ref{tab:ncm-real}.
\begin{table}[h]
	\centering
	\setlength{\tabcolsep}{3pt}
	\caption{Results of NCE problems on real data} \label{tab:ncm-real}
	\begin{tabular}{|r|cccc|cccc|cccc|}
		\hline
		\multirow{2}{*}{matrix} & \multicolumn{4}{c|}{Grad} &\multicolumn{4}{c|}{PFPG} & \multicolumn{4}{c|}{Newton} \\ \cline{2-13} 
		& time & iter & $f$ & $\|g\|$ & time & iter & $f$ & $\|g\|$ & time & iter & $f$ & $\|g\|$ \\ \hline 
		cor1399 &   6.0 &   32 & 6.5e+04 & 1.3e-06 &   2.6 &   32 & 6.5e+04 & 4.1e-06 &   1.9 &    5 & 6.5e+04 & 4.4e-06 \\ \hline
		cor3120 &  76.0 &   43 & 9.3e+04 & 2.4e-06 &  12.3 &   43 & 9.3e+04 & 4.8e-06 &   8.9 &    4 & 9.3e+04 & 3.8e-06 \\ \hline
		bccd16  &  18.9 &   16 & 1.4e+06 & 3.3e-09 &   4.2 &   16 & 1.4e+06 & 6.2e-07 &   5.0 &    2 & 1.4e+06 & 3.4e-07 \\ \hline
	\end{tabular}
\end{table}
For real data, PFPG is still able to perform multi-fold speedups over the traditional gradient method.
Though not always faster than Newton, PFPG is very competitive against the second-order method.

\subsection{Nuclear Norm Minimization}
We next consider two algorithms that deal with the nuclear norm minimization problem (NMM): \textsc{svt} and \textsc{nnls}.

\subsubsection{The SVT Algorithm}
Consider a regularized NMM problem with the form
\begin{equation} \label{eqn:ANM-primal}
	\begin{array}{rl}
	\min & \|X\|_* + \frac{1}{2\mu}\|X\|_F^2, \\
	\st  & \cA(X) = b. \\
	\end{array}
\end{equation}
The dual problem of \eqref{eqn:ANM-primal} is
\begin{equation} \label{eqn:ANM-dual}
	\min_{y \in \bR^m} \frac{1}{2}\|\mathrm{shrink}_\mu(\cA^*(y))\|_F^2 - b^Ty,
\end{equation}
where $\mathrm{shrink}_\mu(X)$ is the shrinkage operator defined on
the matrix space $\bR^{n\times n}$, which actually performs soft-thresholding
on the singular values of $X$, i.e.
\[
\mathrm{shrink}_\mu(X) = U\mathrm{Diag}((\sigma_1 - \mu)_+, \cdots, (\sigma_n - \mu)_+)V^T.
\]
The formulation \eqref{eqn:ANM-dual} is a special case of our
model \eqref{eqn:mat-opt}, with $f$ being the square of a vector shrinkage operator.
Evaluating $f$ and $\nabla f$ only involves computing singular values
which are larger than $\mu$. Thus, problem \eqref{eqn:ANM-dual} satisfies Assumption \ref{assump:f}.

We consider the \textsc{svt} algorithm \cite{svt} which has the iteration scheme
\begin{equation} \label{eqn:svt}
	\left\{\begin{array}{l}
	W^k = \mathrm{shrink}_\mu(\cA^*(x^{k-1})), \\
	x^{k} = x^{k-1} + \tau_k(b - \mathcal{A}(W^k)).\\
	\end{array}
	\right.
\end{equation}
It can be shown that the \textsc{svt} iteration \eqref{eqn:svt} is equivalent to
the gradient descent method on the dual problem \eqref{eqn:ANM-dual}, with
step size $\tau_k$. Thus a polynomial-filtered \textsc{svt} (PFSVT) can be proposed.
A small difference is that we need
to extract a subspace that contains the required singular vectors of a non-symmetric matrix
$A$ for \textsc{svt}. This operation can be performed by the subspace extraction from
$A^TA$ and $AA^T$.

The matrix completion problem is used to test PFSVT, with the constraint
$\cP_{\Omega}(X) = \cP_{\Omega}(M)$.
The test data are generated as follows.
\begin{example} \label{eg:MC}
	The rank $r$ test matrix $M\in\bR^{m\times n}$ is generated randomly 
	using the following procedure: first two random standard Gaussian matrices $M_L\in\bR^{m\times r}$ 
	and $M_R\in\bR^{n\times r}$ are 
	generated and then $M$ is assembled by $M=M_LM_R^T$; the projection set
	$\Omega$ with $p$ index pairs is sampled from $m\times n$ possible pairs uniformly.  We use ``SR'' (sampling ratio) to denote the ratio $p/(mn)$.
\end{example}

We have tested the \textsc{svt} algorithm under three settings: \textsc{svt} with standard Lanczos
SVD (SVT-LANSVD), \textsc{svt} with SVD based on the Gauss-Newton method (SVT-SLRP, proposed in \cite{liu2015efficient}),
and polynomial-filtered \textsc{svt} (PFSVT), where the degree of the polynomial
filter is set to $1$ in all cases. Note that in \eqref{eqn:svt} the variable $W^k$
is always stored in the sparse format. We invoke the \texttt{mkl\_dcsrmm} subroutine in the
Intel Math Kernel Library (MKL) to perform sparse matrix multiplications (SpMM).
The results are displayed in Table \ref{tab:SVT-MKL}, where ``iter'' denotes the outer
iteration of the \textsc{svt} algorithm, ``svp'' stands for the recovered matrix rank
by the algorithm, and ``mse''( = $\|X - M\|_F / \|M\|_F$) denotes the relative mean squared error between the
recovered matrix $X$ and the ground truth $M$.

\begin{table}[h]
	\centering
	\setlength{\tabcolsep}{1pt}
	\caption{Results of \textsc{svt} on randomly generated matrix completion problems by
		using SpMM-MKL} 
	\label{tab:SVT-MKL}
	{%\scriptsize
		\begin{tabular}{|ccc|cccc|cccc|cccc|} \hline  
			\multicolumn{3}{|c|}{}     & \multicolumn{4}{c|}{SVT-LANSVD} &
			\multicolumn{4}{c|}{SVT-SLRP} & \multicolumn{4}{c|}{PFSVT} \\ \hline
			m=n & SR (\%)& r  & iter &  svp & time  &  mse & iter &  svp & time  &  mse  & iter &  svp & time  &  mse \\ \hline
			1000  & 20.0 &   10  & 79 & 10 & 4.23 & 1.31e-04  & 79 & 10 & 1.94 & 1.31e-04  & 78 & 10 & 1.83 & 1.40e-04 \\ \hline
			1000  & 30.0 &   10  & 62 & 10 & 3.35 & 1.17e-04  & 62 & 10 & 1.71 & 1.17e-04  & 62 & 10 & 1.66 & 1.15e-04 \\ \hline
			1000  & 40.0 &   10  & 53 & 10 & 2.94 & 1.15e-04  & 53 & 10 & 1.52 & 1.14e-04  & 53 & 10 & 1.79 & 1.13e-04 \\ \hline
			1000  & 30.0 &   50  & 169 & 69 & 43.58 & 2.12e-04  & 171 & 68 & 8.67 & 2.12e-04  & 168 & 67 & 7.31 & 2.10e-04 \\ \hline
			1000  & 40.0 &   50  & 110 & 50 & 13.77 & 1.60e-04  & 110 & 50 & 5.35 & 1.60e-04  & 110 & 50 & 5.28 & 1.58e-04 \\ \hline
			1000  & 50.0 &   50  & 86 & 50 & 13.01 & 1.40e-04  & 86 & 50 & 4.73 & 1.40e-04  & 86 & 50 & 4.38 & 1.39e-04 \\ \hline
			5000  & 10.0 &   10  & 53 & 10 & 28.88 & 1.08e-04  & 53 & 10 & 12.37 & 1.08e-04  & 52 & 10 & 12.47 & 1.22e-04 \\ \hline
			5000  & 20.0 &   10  & 42 & 10 & 30.43 & 1.04e-04  & 42 & 10 & 12.85 & 1.04e-04  & 42 & 10 & 13.39 & 1.05e-04 \\ \hline
			5000  & 30.0 &   10  & 38 & 10 & 34.16 & 9.23e-05  & 38 & 10 & 16.62 & 9.23e-05  & 38 & 10 & 16.99 & 9.55e-05 \\ \hline
			5000  & 10.0 &   50  & 107 & 50 & 107.38 & 1.59e-04  & 107 & 50 & 64.38 & 1.59e-04  & 107 & 50 & 62.10 & 1.60e-04 \\ \hline
			5000  & 20.0 &   50  & 67 & 50 & 90.14 & 1.23e-04  & 67 & 50 & 33.11 & 1.22e-04  & 67 & 50 & 27.83 & 1.25e-04 \\ \hline
			5000  & 30.0 &   50  & 54 & 50 & 92.55 & 1.21e-04  & 54 & 50 & 37.90 & 1.21e-04  & 54 & 50 & 32.87 & 1.21e-04 \\ \hline
			10000  & 5.0 &   10  & 54 & 10 & 59.63 & 1.10e-04  & 54 & 10 & 27.41 & 1.10e-04  & 53 & 10 & 27.12 & 1.23e-04 \\ \hline
			10000  & 8.0 &   10  & 46 & 10 & 82.54 & 1.08e-04  & 46 & 10 & 34.29 & 1.08e-04  & 46 & 10 & 35.20 & 1.08e-04 \\ \hline
			10000  & 10.0 &   10  & 43 & 10 & 84.65 & 1.07e-04  & 43 & 10 & 39.26 & 1.07e-04  & 43 & 10 & 39.52 & 1.08e-04 \\ \hline
			10000  & 5.0 &   50  & 109 & 50 & 271.86 & 1.66e-04  & 109 & 50 & 150.35 & 1.65e-04  & 109 & 50 & 131.95 & 1.67e-04 \\ \hline
			10000  & 10.0 &   50  & 69 & 50 & 260.60 & 1.29e-04  & 69 & 50 & 173.75 & 1.29e-04  & 69 & 50 & 152.81 & 1.30e-04 \\ \hline
			10000  & 15.0 &   50  & 57 & 50 & 291.30 & 1.16e-04  & 57 & 50 & 206.06 & 1.16e-04  & 57 & 50 & 178.84 & 1.22e-04 \\ \hline 
			
		\end{tabular}
	}
\end{table}
The following observations can be drawn from Table \ref{tab:SVT-MKL}.
\begin{itemize}
	\item All three solvers produce similar results: ``iter'', ``svp'', and ``mse''
	are almost the same.
	\item SVT-SLRP and PFSVT have similar speed, both of which enjoy 1 to 5 times
	speedup over SVT-LANSVD.  In all test cases SVT-SLRP is a little bit slower than
	PFSVT, yet still comparable.
\end{itemize}

\subsubsection{The NNLS Algorithm}
Consider a penalized formulation of NMM:
\begin{equation} \label{eqn:NMM-penalized}
	\min \|X\|_* + \frac{1}{2\mu}\|\cA(X) - b\|_F^2.
\end{equation}
We turn to the \textsc{nnls} algorithm \cite{toh2010accelerated} to solve \eqref{eqn:NMM-penalized},
which is essentially an accelerated proximal gradient method. At the $k$-th iteration,
 the main cost is to compute the truncated SVD of a matrix
\[
A^k = \beta_1 U^k(V^k)^T - \beta_2U^{k-1}(V^{k-1})^T - \beta_3 G^k,
\]
where $U^k,V^k,U^{k-1},V^{k-1}$ are dense matrices, but the matrix $G^k$
is either dense or sparse, depending on the sample ratio (SR). Though
the formulation \eqref{eqn:NMM-penalized} is not a special case of \eqref{eqn:mat-opt},
we can still insert the polynomial filter \eqref{eqn:subspace-update} into 
the \textsc{nnls} algorithm to reduce the cost of SVD, resulting in
a polynomial-filtered \textsc{nnls} algorithm (PFNNLS).
The Example \ref{eg:MC} is still used to compare the performance of
NNLS-LANSVD, NNLS-SLRP, and PFNNLS. The polynomial filter degree is merely set to 1.
The results are shown in Table \ref{tab:NNLS-dense}
for dense $G^k$ and Table \ref{tab:NNLS-sparse-MKL} for sparse $G^k$. For matrix multiplications, 
we call \textsc{matlab} built-in functions for dense $G^k$ and Intel MKL subroutines for sparse $G^k$.
The notations have the same meaning as in Table \ref{tab:SVT-MKL}.

\begin{table}[h]
	\centering
	\setlength{\tabcolsep}{2pt}
	\caption{Results of \textsc{nnls} on dense examples} 
	\label{tab:NNLS-dense}
	{%\scriptsize
		\begin{tabular}{|ccc|cccc|cccc|cccc|} \hline  
			\multicolumn{3}{|c|}{}     & \multicolumn{4}{c|}{NNLS-LANSVD} &
			\multicolumn{4}{c|}{NNLS-SLRP} & \multicolumn{4}{c|}{PFNNLS} \\ \hline
			m=n & SR (\%)& r  & iter &  svp & time  &  mse & iter &  svp & time  &  mse  & iter &  svp & time  &  mse \\ \hline
			1000 & 25.0 & 50  & 55 & 50 & 5.1 & 7.60e-04 & 55 & 50 & 2.7 & 1.18e-03 & 54 & 50 & 2.2 & 1.19e-03\\ \hline 
			1000 & 35.0 & 50  & 42 & 50 & 3.1 & 5.42e-04 & 49 & 50 & 2.6 & 4.40e-04 & 49 & 50 & 1.9 & 4.66e-04\\ \hline 
			1000 & 49.9 & 50  & 39 & 50 & 3.2 & 1.81e-04 & 39 & 50 & 1.9 & 1.81e-04 & 39 & 50 & 1.6 & 1.84e-04\\ \hline 
			1000 & 25.0 & 100  & 100 & 150 & 29.4 & 3.55e-01 & 100 & 150 & 8.6 & 3.87e-01 & 100 & 150 & 7.1 & 5.03e-01\\ \hline 
			1000 & 35.0 & 100  & 64 & 100 & 10.7 & 1.45e-03 & 64 & 100 & 4.2 & 1.48e-03 & 65 & 100 & 3.3 & 9.32e-04\\ \hline 
			1000 & 49.9 & 100  & 54 & 100 & 8.3 & 4.70e-04 & 48 & 100 & 3.8 & 1.05e-03 & 51 & 100 & 3.2 & 4.33e-04\\ \hline 
			5000 & 25.0 & 50  & 36 & 50 & 40.0 & 1.41e-04 & 36 & 50 & 17.8 & 1.42e-04 & 36 & 50 & 16.1 & 1.46e-04\\ \hline 
			5000 & 35.0 & 50  & 34 & 50 & 40.0 & 1.60e-04 & 34 & 50 & 18.7 & 1.61e-04 & 34 & 50 & 17.7 & 1.61e-04\\ \hline 
			5000 & 50.0 & 50  & 32 & 50 & 41.7 & 1.46e-04 & 32 & 50 & 21.5 & 1.48e-04 & 32 & 50 & 20.6 & 1.45e-04\\ \hline 
			5000 & 25.0 & 100  & 49 & 100 & 106.4 & 2.83e-04 & 49 & 100 & 35.6 & 2.88e-04 & 50 & 100 & 32.4 & 2.19e-04\\ \hline 
			5000 & 35.0 & 100  & 45 & 100 & 94.0 & 1.74e-04 & 45 & 100 & 38.7 & 1.76e-04 & 45 & 100 & 35.2 & 1.74e-04\\ \hline 
			5000 & 50.0 & 100  & 43 & 100 & 100.5 & 1.43e-04 & 43 & 100 & 46.1 & 1.44e-04 & 43 & 100 & 41.3 & 1.45e-04\\ \hline 
			10000 & 25.0 & 50  & 32 & 50 & 106.9 & 1.17e-04 & 32 & 50 & 54.9 & 1.17e-04 & 32 & 50 & 51.4 & 1.18e-04\\ \hline 
			10000 & 35.0 & 50  & 32 & 50 & 120.0 & 1.16e-04 & 32 & 50 & 66.3 & 1.16e-04 & 32 & 50 & 61.9 & 1.17e-04\\ \hline 
			10000 & 50.0 & 50  & 31 & 50 & 129.1 & 1.39e-04 & 32 & 50 & 83.0 & 1.15e-04 & 31 & 50 & 76.2 & 1.37e-04\\ \hline 
			10000 & 25.0 & 100  & 43 & 100 & 253.0 & 1.59e-04 & 44 & 100 & 114.9 & 1.74e-04 & 45 & 100 & 105.0 & 1.56e-04\\ \hline 
			10000 & 35.0 & 100  & 42 & 100 & 263.7 & 1.32e-04 & 42 & 100 & 130.1 & 1.33e-04 & 42 & 100 & 118.3 & 1.37e-04\\ \hline 
			10000 & 50.0 & 100  & 39 & 100 & 259.6 & 3.70e-04 & 40 & 100 & 156.1 & 1.18e-04 & 39 & 100 & 136.5 & 3.56e-04\\ \hline 
			
		\end{tabular}
	}
\end{table}

\begin{table}[h]
	\centering
	\setlength{\tabcolsep}{2pt}
	\caption{Results of \textsc{nnls} on random sparse examples by using SpMM-MKL} 
	\label{tab:NNLS-sparse-MKL}
	{%\scriptsize
		\begin{tabular}{|ccc|cccc|cccc|cccc|} \hline  
			\multicolumn{3}{|c|}{}     & \multicolumn{4}{c|}{NNLS-LANSVD} &
			\multicolumn{4}{c|}{NNLS-SLRP} & \multicolumn{4}{c|}{PFNNLS} \\ \hline
			m=n & SR (\%)& r  & iter &  svp & time  &  mse & iter &  svp & time  &  mse  & iter &  svp & time  &  mse \\ \hline
			10000 & 2.0 & 10 & 45 & 10 & 16.5 & 1.86e-04 & 45 & 10 & 9.2 & 1.54e-04 & 45 & 10 & 7.9 & 1.55e-04\\ \hline 
			10000 & 5.0 & 10 & 35 & 10 & 27.9 & 1.30e-04 & 35 & 10 & 15.0 & 1.30e-04 & 35 & 10 & 12.9 & 1.30e-04\\ \hline 
			10000 & 10.0 & 10 & 30 & 10 & 53.1 & 1.03e-04 & 30 & 10 & 23.9 & 1.03e-04 & 30 & 10 & 18.9 & 1.03e-04\\ \hline 
			10000 & 14.0 & 10 & 28 & 10 & 62.6 & 1.07e-04 & 29 & 10 & 28.6 & 1.09e-04 & 28 & 10 & 22.7 & 1.07e-04\\ \hline 
			10000 & 2.0 & 50 & 100 & 77 & 226.7 & 3.42e-01 & 100 & 50 & 53.9 & 9.88e-03 & 100 & 74 & 43.5 & 7.48e-01\\ \hline 
			10000 & 5.0 & 50 & 50 & 50 & 125.3 & 2.50e-04 & 48 & 50 & 40.1 & 2.52e-04 & 51 & 50 & 33.7 & 1.72e-04\\ \hline 
			10000 & 10.0 & 50 & 37 & 50 & 187.3 & 1.66e-04 & 37 & 50 & 53.0 & 1.67e-04 & 37 & 50 & 46.3 & 1.79e-04\\ \hline 
			10000 & 14.0 & 50 & 35 & 50 & 224.2 & 1.29e-04 & 35 & 50 & 62.7 & 1.30e-04 & 35 & 50 & 53.4 & 1.31e-04\\ \hline 
			50000 & 2.0 & 10 & 38 & 10 & 431.9 & 1.09e-04 & 38 & 10 & 168.3 & 1.09e-04 & 38 & 10 & 137.7 & 1.10e-04\\ \hline 
			50000 & 5.0 & 10 & 31 & 10 & 776.4 & 9.62e-05 & 31 & 10 & 328.7 & 9.62e-05 & 31 & 10 & 253.9 & 9.58e-05\\ \hline 
			50000 & 10.0 & 10 & 28 & 10 & 1560.0 & 1.09e-04 & 28 & 10 & 586.3 & 1.09e-04 & 28 & 10 & 438.0 & 1.02e-04\\ \hline 
			50000 & 14.0 & 10 & 28 & 10 & 2011.5 & 1.03e-04 & 28 & 10 & 792.0 & 1.03e-04 & 28 & 10 & 611.2 & 1.03e-04\\ \hline 
			50000 & 2.0 & 50 & 49 & 50 & 1415.8 & 1.96e-04 & 49 & 50 & 442.0 & 1.62e-04 & 52 & 50 & 335.4 & 1.38e-04\\ \hline 
			50000 & 5.0 & 50 & 37 & 50 & 2487.5 & 1.74e-04 & 37 & 50 & 710.7 & 1.77e-04 & 37 & 50 & 600.1 & 1.99e-04\\ \hline 
			50000 & 10.0 & 50 & 32 & 50 & 4483.5 & 1.17e-04 & 32 & 50 & 1244.8 & 1.17e-04 & 32 & 50 & 1074.5 & 1.17e-04\\ \hline 
			50000 & 14.0 & 50 & 31 & 50 & 5687.1 & 1.10e-04 & 31 & 50 & 1499.5 & 1.10e-04 & 31 & 50 & 1405.2 & 1.10e-04\\ \hline 
		\end{tabular}
	}
\end{table}

Similar observations can be drawn from Table \ref{tab:NNLS-dense} and \ref{tab:NNLS-sparse-MKL}.
However, for one case (sparse $G^k$, $m=n=10000$, SR=2.0\%), NNLS-LANSVD and PFNNLS seem to
produce different results from NNLS-SLRP. In another case (dense $G^k$, $m=n=1000$, SR=25.0\%)
the mse of all three methods is high. The reason is that \textsc{nnls} terminates after it has
reached the max iteration 100.

The last experiment is to use the three \textsc{nnls} algorithms on real data sets:
the Jester joke data set and the MovieLens data set, which are also mentioned in \cite{toh2010accelerated}.
As shown in Table \ref{tab:NNLS-real-MKL}, the Jester joke data set includes four examples:
jester-1, jester-2, jester-3, and jester-all. The last one is simply the combination of the
first three data sets. The MovieLens data set have three problems based on the size:
movie-100K, movie-1M, and movie-10M. We call Intel MKL routines for matrix multiplications
since the data are all sparse. Again, the degree of the polynomial filter is set to 1 and
the max iteration of \textsc{nnls} is 100.

%\begin{table}[h]
%	\centering
%	\setlength{\tabcolsep}{1.5pt}
%	\caption{Results of \textsc{nnls} on real examples by using SpMM-Matlab} 
%	\label{tab:NNLS-real-matlab}
%	\begin{tabular}{|cc|cccc|cccc|cccc|} \hline
%		\multicolumn{2}{|c|}{}     & \multicolumn{4}{c|}{NNLS-LANSVD} &
%		\multicolumn{4}{c|}{NNLS-SLRP} & \multicolumn{4}{c|}{NNLS-CheSE} \\ \hline
%		name & $(m, n)$ & iter &  svp & time  &  mse & iter &  svp & time  &  mse  & iter &  svp & time  &  mse \\ \hline
%		jester-1 & (24983,100) & 26 & 93 & 9.7 & 1.64e-01 & 27 & 69 & 6.6 & 1.76e-01 & 24 & 84 & 3.0 & 1.80e-01 \\ \hline
%		jester-2 & (23500,100) & 26 & 93 & 8.3 & 1.65e-01 & 26 & 79 & 6.3 & 1.72e-01 & 25 & 88 & 3.2 & 1.80e-01 \\ \hline
%		jester-3 & (24938,100) & 24 & 83 & 6.4 & 1.16e-01 & 27 & 74 & 5.6 & 1.24e-01 & 25 & 82 & 2.9 & 1.32e-01 \\ \hline
%		jester-all & (73421,100) & 26 & 93 & 25.0 & 1.58e-01 & 26 & 82 & 20.1 & 1.62e-01 & 25 & 88 & 9.6 & 1.75e-01 \\ \hline
%		moive-100K & (943, 1682) & 34 & 100 & 4.3 & 1.28e-01 & 35 & 100 & 1.7 & 1.26e-01 & 36 & 100 & 1.1 & 1.23e-01 \\ \hline
%		moive-1M & (6040,3706) & 50 & 100 & 37.9 & 1.42e-01 & 50 & 100 & 16.1 & 1.43e-01 & 51 & 100 & 10.6 & 1.42e-01 \\ \hline 
%		moive-10M & (71567,10677) & 54 & 100 & 418.7 & 1.26e-01 & 57 & 100 & 361.9 & 1.27e-01 & 52 & 100 & 184.8 & 1.27e-01 \\ \hline
%		
%	\end{tabular}
%\end{table}

\begin{table}[h]
	\centering
	\setlength{\tabcolsep}{1.5pt}
	\caption{Results of \textsc{nnls} on real examples by using SpMM-MKL} 
	\label{tab:NNLS-real-MKL}
	\begin{tabular}{|cc|cccc|cccc|cccc|} \hline
		\multicolumn{2}{|c|}{}     & \multicolumn{4}{c|}{NNLS-LANSVD} &
		\multicolumn{4}{c|}{NNLS-SLRP} & \multicolumn{4}{c|}{PFNNLS} \\ \hline
		name & $(m,n)$ & iter &  svp & time  &  mse & iter &  svp & time  &  mse  & iter &  svp & time  &  mse \\ \hline
		jester-1 & (24983, 100)  & 26 & 93 & 10.5 & 1.64e-01 & 27 & 69 & 4.6 & 1.76e-01 & 24 & 84 & 2.3 & 1.80e-01\\ \hline
		jester-2 & (23500, 100)  & 26 & 93 & 9.1 & 1.65e-01 & 26 & 79 & 4.3 & 1.72e-01 & 25 & 88 & 2.1 & 1.80e-01\\ \hline
		jester-3 & (24938, 100)  & 24 & 83 & 7.1 & 1.16e-01 & 27 & 74 & 4.6 & 1.24e-01 & 24 & 84 & 2.0 & 1.30e-01\\ \hline
		jester-all & (73421, 100)  & 26 & 93 & 26.2 & 1.58e-01 & 26 & 82 & 12.4 & 1.62e-01 & 24 & 84 & 5.6 & 1.74e-01\\ \hline
		moive-100K & (943, 1682)  & 34 & 100 & 4.2 & 1.28e-01 & 35 & 100 & 0.8 & 1.26e-01 & 36 & 100 & 0.6 & 1.23e-01\\ \hline
		moive-1M & (6040, 3706)  & 50 & 100 & 40.6 & 1.42e-01 & 50 & 100 & 10.8 & 1.43e-01 & 51 & 100 & 7.4 & 1.42e-01\\ \hline
		moive-10M & (71567, 10677)  & 54 & 100 & 620.1 & 1.26e-01 & 57 & 100 & 179.9 & 1.27e-01 & 52 & 100 & 92.7 & 1.27e-01\\ \hline
	\end{tabular}
\end{table}

We have the following observations from Table \ref{tab:NNLS-real-MKL}:
\begin{itemize}
	\item The three methods require similar number of iterations for all seven test cases.
	\item NNLS-LANSVD generates solutions with the highest rank, while NNLS-SLRP produces
	solutions with the lowest rank.
	The mse of PFNNLS is a little larger in Jester joke data sets.
	\item NNLS-SLRP provides 2-4 speedups over NNLS-LANSVD, while PFNNLS is 3-6 times
	faster than NNLS-LANSVD.
\end{itemize}

We make the following comments as extra interpretations of all numerical results
in this subsection.
\begin{itemize}
	\item The SLRP solver is essentially a subspace SVD method.
	As the authors point out in \cite{liu2015efficient}, the subspace is generated by solving a least square model
	by one Gauss-Newton iteration, after which high accuracy singular pairs can be extracted
	from the subspace. The difference of SLRP and our polynomial filters is that
	solving the SLRP sub-problem requires more computation than evaluating polynomial
	filters, in exchange for higher precision solutions. Hence SLRP can be regarded
	as a variant of the subspace extraction method.
	\item Converting the SVD of $A$ to the EVD of $A^TA$ implicitly applies
	a polynomial filter $\rho(t) = t^2$ to the singular values. Thus we can
	simply let the degree equal to 1 when performing subspace extractions on $A^TA$.
\end{itemize}
\subsection{Max Eigenvalue Optimization}
Consider the max eigenvalue optimization problem with
the following form
\begin{equation} \label{eqn:maxeig}
	\min_x F(x) + R(x) := \lambda_1(\cB(x)) + R(x),
\end{equation}
where $\cB(x) = G + \mathcal{A}^*(x)$. The subgradient
of $F(x)$ is
\[
\partial F(x) = \{\cA(U_1SU_1^T)~|~ S \succeq 0, \tr(S) = 1 \},
\]
where $U_1 \in \bR^{n\times r_1}$ is the subspace spanned by eigenvectors
of $\lambda_1(\cB(x))$ with multiplicity $r_1$. If $r_1 = 1$, then $\partial F(x)$ only contains one
element hence the subgradient becomes the gradient. If $r_1 > 1$, $F(x)$ is only sub-differentiable.
We point out that the $r_1 > 1$ case is generally harder than $r_1 = 1$.
In the following two applications, namely, phase recovery and blind deconvolution,
$F(x)$ is differentiable within a neighborhood of $x^*$.
\subsubsection{Phase Recovery}
We use the phase recovery formulation in \cite{doi:10.1137/15M1034283}.
The smooth part $F(x)=\lambda_1(\cB(x)):=\lambda_1(\cA^*(x))$ is defined
as follows. Let the known diagonal matrices
$C_k\in\bC^{n\times n}$ be the encoding diffraction patterns
corresponding to the $k$-th ``mask'' ($k=1,\ldots,L$). The measurements of an unknown signal $x_0 \in \bC^n$ are given
by
\[
b = \cA(x_{0}x_{0}^{*}):=\diag\left[
\begin{pmatrix}\cF C_{1}\\\vdots\\\cF C_L\end{pmatrix}
(x_{0}x_{0}^{*})
\begin{pmatrix}\cF C_{1}\\\vdots\\\cF C_L\end{pmatrix}^{\!\!*\,}
\right],
\]
where $\cF$ is the unitary discrete Fourier transform (DFT).  The
adjoint of $\cA$ can be written as
\[
 \cA^{*}y := \sum_{k=1}^{L} C_k^* \cF^* \mathrm{Diag}(y) \cF C_k.
\]
The non-smooth part $R(x)$ defined as the indicator function of the set
\[
\{x~|~\iprod{b}{x} - \|x\|_* \geq 1 \},
\]
where $\|\cdot\|_*$ is the dual norm of $\|\cdot \|_2$.
In the experiments of the phase retrieval problem,
both synthetic and real data are considered.
The noiseless synthetic data are generated as follows.
For each value of $L=7,\ldots,12$, we generate
random complex Gaussian vectors $x_{0}$ of length $n=1024,4096$, and
a set of $L$ random complex Gaussian masks $C_{k}$. 

We use the GAUGE algorithm proposed in \cite{doi:10.1137/15M1034283} to solve the phase recovery problem,
in which the \textsc{matlab} built-in function \texttt{eigs} is used as
the eigenvalue sub-problem solver.
Then we insert our polynomial filters to the original GAUGE implementation to
obtain the polynomial-filtered GAUGE algorithm (PFGAUGE). It should be mentioned
that though GAUGE is essentially a gradient method, it has many variants
to handle different cases, which we will briefly discuss about later.
A big highlight of GAUGE is a so-called ``primal-dual refinement step'', which
is able to reduce the iteration steps drastically. This feature also
has great impact on the performance of GAUGE and PFGAUGE.

The results of noiseless synthetic data is presented in Table \ref{tab:noiseless-1024}
and \ref{tab:noiseless-4096}, where ``iter'' denotes the number of outer iteration.
In addition, we record the number of DFT calls (nDFT) and the number of eigen-solvers or polynomial filters
evaluated, which is denoted by nEigs and nPF, respectively. The label ``err'' stands
for the relative mean squared error of the solution $\|xx^* - x_0x_0^*\|_F / \|x_0\|_2$.

\begin{table}[h]
	\centering
	\caption{Random Gaussian signals, noiseless, $n = 1024$} \label{tab:noiseless-1024}
	\begin{tabular}{|c|cccc|cccc|} \hline
		& \multicolumn{4}{c|}{GAUGE} & \multicolumn{4}{c|}{PFGAUGE} \\ \hline
		L & time & iter & nDFT(nEigs) & err & time & iter & nDFT(nPF) & err \\ \hline
		12 & 2.37 & 3 & 34248(10) & 8.7e-07 & 2.09 & 7 & 43386(17) & 1.1e-06 \\ \hline
		11 & 2.51 & 4 & 46244(13) & 9.4e-07 & 2.28 & 9 & 48813(21) & 1.2e-06 \\ \hline
		10 & 3.81 & 6 & 62160(17) & 7.0e-07 & 2.54 & 10 & 51900(25) & 1.3e-06 \\ \hline
		9 & 6.53 & 8 & 74772(19) & 7.9e-07 & 3.74 & 12 & 56030(28) & 1.8e-06 \\ \hline
		8 & 10.09 & 13 & 117776(29) & 1.2e-06 & 5.00 & 16 & 67960(36) & 1.6e-06 \\ \hline
		7 & 17.72 & 21 & 193956(47) & 1.6e-06 & 6.92 & 23 & 88785(50) & 6.8e-07 \\ \hline
		%6 & \multicolumn{4}{c|}{EIGS FAILED} & 45.24 & 112 & 512214(238) & 2.0e-06 \\ \hline
	\end{tabular}
\end{table}

\begin{table}[h]
	\centering
	\caption{Random Gaussian signals, noiseless, $n = 4096$} \label{tab:noiseless-4096}
	\begin{tabular}{|c|cccc|cccc|} \hline
		& \multicolumn{4}{c|}{GAUGE} & \multicolumn{4}{c|}{PFGAUGE} \\ \hline
		L & time & iter & nDFT(nEigs) & err & time & iter & nDFT(nPF) & err \\ \hline
		12 & 10.85 & 4 & 57966(12) & 1.4e-06 & 9.66 & 11 & 74442(30) & 2.0e-06 \\ \hline
		11 & 15.25 & 5 & 83298(16) & 2.0e-06 & 10.43 & 11 & 77880(36) & 2.0e-06 \\ \hline
		10 & 22.63 & 8 & 118530(20) & 2.3e-06 & 11.20 & 14 & 79435(37) & 2.6e-06 \\ \hline
		9 & 35.89 & 11 & 176864(26) & 2.0e-06 & 16.78 & 20 & 118377(56) & 3.4e-06 \\ \hline
		8 & 49.88 & 15 & 235720(36) & 2.8e-06 & 15.52 & 23 & 105660(53) & 3.0e-06 \\ \hline
		7 & 160.64 & 40 & 779559(83) & 2.6e-06 & 33.66 & 40 & 208915(84) & 3.4e-06 \\ \hline
		%6 & \multicolumn{4}{c|}{EIGS FAILED} & \multicolumn{4}{c|}{EIGS FAILED} \\ \hline
	\end{tabular}
\end{table}

The following observations should be clear from the results.
\begin{itemize}
	\item Both GAUGE and PFGAUGE obtain the required accuracy in all
	test cases. The problem grows harder to solve as $L$ decreases from 12 to 7.
	\item In general, it takes more iterations for PFGAUGE to converge compared
	with the original GAUGE. An exception is the $L=7$ case, in which the
	iterations are nearly the same.
	\item PFGAUGE gradually becomes faster as $L$ decreases. For example, it
	gives 5 times speedup at $L=7$ and $n=4096$. This can be also concluded from nDFT.
	After examining the iteration process, we find out that the eigenvalue sub-problem
	is hard to solve when $L$ is small, thus it takes more iterations for \texttt{eigs}
	to converge. However, since we only apply a polynomial filter to the iteration point,
	the cost of DFT can be greatly reduced. For the easy cases such as $L=12$, the
	performances of the two solvers are similar.
\end{itemize}

We next consider noisy synthetic data, which is generated with the following steps.
First we generate $L$ octanary masks $C_k$ as described in \cite{2014arXiv1407.1065C}, and
a real Gaussian vector $y\in \bR^n$. Then a solution $x_0 \in \bC^n$ is
chosen as the normalized rightmost eigenvector of $\cA^*(y)$, where $n=1024,4096$
as mentioned before. The
measurement vector $b$ is computed as $b:= \cA(x_0x_0^*) + \epsilon y/\|y\|_2$,
where $\epsilon:=\|b - \cA(x_0x_0^*)\|_2=\eta\|b\|_2$ is related with
a given relative noise level $\eta \in (0, 1)$.
The numerical experiments are conducted in the following procedure: first we choose the number of
masks $L\in\{12,9,6\}$ and noise level $\eta\in\{0.1\%,0.5\%,1\%,5\%, 10\%, 50\%\}$.
For each set of parameter, we randomly generate 100 test instances and feed them
to the solvers, after which we compute the relative mean squared error of the solution.
The ``\%'' label in Table \ref{tab:noisy-1024} and \ref{tab:noisy-4096} stands
for the rate of successful recoveries, which is defined as the number of solutions
whose MSE is smaller than $10^{-2}$. All other columns in Table \ref{tab:noisy-1024} and
\ref{tab:noisy-4096} record the average value of the 100 test instances.

\begin{table}[h]
	\centering
	\setlength{\tabcolsep}{3pt}
	\caption{Random Gaussian signals, noisy, $n = 1024$} \label{tab:noisy-1024}
	\begin{tabular}{|c|c|ccccc|ccccc|} \hline
		& & \multicolumn{5}{c|}{GAUGE} & \multicolumn{5}{c|}{PFGAUGE} \\ \hline
		L & $\eta$ & time & iter & nDFT(nEigs) & \% & err & time & iter & nDFT(nPF) & \% & err \\ \hline
		12 & 0.1\% & 3.02 & 51 & 9e+04(58) & 100 & 2.1e-03 & 2.53 & 48 & 1e+05(55) & 100 & 2.8e-03 \\ \hline
		9 & 0.1\% & 6.47 & 88 & 1e+05(102) & 96 & 3.4e-03 & 4.95 & 88 & 2e+05(103) & 98 & 3.8e-03 \\ \hline
		6 & 0.1\% & 13.27 & 195 & 2e+05(224) & 76 & 7.5e-03 & 9.14 & 198 & 2e+05(223) & 76 & 7.6e-03 \\ \hline
		12 & 0.5\% & 5.76 & 94 & 2e+05(107) & 95 & 5.5e-03 & 4.62 & 92 & 2e+05(99) & 94 & 5.8e-03 \\ \hline
		9 & 0.5\% & 9.90 & 128 & 2e+05(144) & 92 & 5.7e-03 & 6.78 & 127 & 2e+05(147) & 93 & 5.6e-03 \\ \hline
		6 & 0.5\% & 19.84 & 284 & 3e+05(327) & 74 & 8.5e-03 & 12.77 & 284 & 3e+05(335) & 68 & 8.5e-03 \\ \hline
		12 & 1.0\% & 22.82 & 296 & 7e+05(354) & 78 & 7.2e-03 & 14.80 & 326 & 7e+05(372) & 83 & 7.2e-03 \\ \hline
		9 & 1.0\% & 22.26 & 271 & 4e+05(318) & 86 & 8.0e-03 & 13.92 & 290 & 5e+05(346) & 88 & 7.8e-03 \\ \hline
		6 & 1.0\% & 71.67 & 979 & 1e+06(1171) & 92 & 6.5e-03 & 38.64 & 1049 & 1e+06(1287) & 92 & 6.4e-03 \\ \hline
		12 & 5.0\% & 43.27 & 556 & 1e+06(649) & 34 & 1.2e-02 & 30.53 & 734 & 1e+06(853) & 31 & 1.4e-02 \\ \hline
		9 & 5.0\% & 86.61 & 836 & 2e+06(1005) & 65 & 7.8e-03 & 37.42 & 1009 & 1e+06(1161) & 53 & 9.2e-03 \\ \hline
		6 & 5.0\% & 90.90 & 1113 & 1e+06(1322) & 88 & 4.7e-03 & 43.35 & 1237 & 1e+06(1477) & 81 & 5.8e-03 \\ \hline
		12 & 10.0\% & 61.22 & 692 & 2e+06(819) & 21 & 1.6e-02 & 29.86 & 749 & 1e+06(842) & 23 & 1.5e-02 \\ \hline
		9 & 10.0\% & 76.80 & 686 & 1e+06(804) & 48 & 1.0e-02 & 31.81 & 799 & 1e+06(926) & 50 & 1.0e-02 \\ \hline
		6 & 10.0\% & 78.61 & 950 & 1e+06(1124) & 76 & 5.3e-03 & 32.38 & 956 & 1e+06(1110) & 71 & 6.0e-03 \\ \hline
		12 & 50.0\% & 36.20 & 368 & 1e+06(433) & 41 & 1.3e-02 & 17.04 & 424 & 8e+05(491) & 55 & 8.9e-03 \\ \hline
		9 & 50.0\% & 42.28 & 361 & 8e+05(424) & 68 & 5.7e-03 & 14.76 & 374 & 5e+05(438) & 56 & 6.6e-03 \\ \hline
		6 & 50.0\% & 34.76 & 389 & 6e+05(463) & 86 & 2.5e-03 & 14.37 & 391 & 4e+05(477) & 87 & 2.5e-03 \\ \hline
		
	\end{tabular}
\end{table}

\begin{table}[h]
	\centering
	\setlength{\tabcolsep}{3pt}
	\caption{Random Gaussian signals, noisy, $n = 4096$} \label{tab:noisy-4096}
	\begin{tabular}{|c|c|ccccc|ccccc|} \hline
		& & \multicolumn{5}{c|}{GAUGE} & \multicolumn{5}{c|}{PFGAUGE} \\ \hline
		L & $\eta$ & time & iter & nDFT(nEigs) & \% & err & time & iter & nDFT(nPF) & \% & err \\ \hline
		12 & 0.1\% & 26.16 & 90 & 2e+05(106) & 89 & 3.5e-03 & 20.90 & 89 & 2e+05(104) & 88 & 3.1e-03 \\ \hline
		9 & 0.1\% & 45.83 & 194 & 4e+05(224) & 84 & 5.2e-03 & 30.10 & 192 & 3e+05(221) & 80 & 4.6e-03 \\ \hline
		6 & 0.1\% & 70.11 & 373 & 5e+05(436) & 41 & 1.1e-02 & 56.40 & 407 & 5e+05(487) & 35 & 1.3e-02 \\ \hline
		12 & 0.5\% & 36.31 & 124 & 3e+05(149) & 85 & 6.2e-03 & 29.69 & 138 & 3e+05(161) & 90 & 5.3e-03 \\ \hline
		9 & 0.5\% & 68.44 & 275 & 5e+05(310) & 72 & 7.9e-03 & 43.98 & 266 & 5e+05(319) & 69 & 8.0e-03 \\ \hline
		6 & 0.5\% & 117.91 & 625 & 8e+05(720) & 24 & 1.4e-02 & 84.42 & 653 & 7e+05(757) & 19 & 1.5e-02 \\ \hline
		12 & 1.0\% & 98.84 & 290 & 8e+05(331) & 73 & 7.6e-03 & 64.63 & 320 & 8e+05(409) & 76 & 7.2e-03 \\ \hline
		9 & 1.0\% & 109.12 & 404 & 9e+05(471) & 53 & 9.7e-03 & 62.14 & 473 & 8e+05(561) & 52 & 9.9e-03 \\ \hline
		6 & 1.0\% & 174.41 & 900 & 1e+06(1026) & 25 & 1.4e-02 & 99.72 & 813 & 9e+05(951) & 13 & 1.6e-02 \\ \hline
		12 & 5.0\% & 378.11 & 834 & 3e+06(945) & 54 & 9.5e-03 & 128.56 & 939 & 2e+06(1078) & 45 & 1.1e-02 \\ \hline
		9 & 5.0\% & 211.98 & 705 & 2e+06(831) & 36 & 1.3e-02 & 94.93 & 793 & 1e+06(912) & 36 & 1.3e-02 \\ \hline
		6 & 5.0\% & 211.82 & 937 & 1e+06(1084) & 29 & 1.5e-02 & 119.41 & 1048 & 1e+06(1284) & 18 & 1.7e-02 \\ \hline
		12 & 10.0\% & 372.63 & 801 & 3e+06(935) & 35 & 1.4e-02 & 118.20 & 896 & 2e+06(1036) & 34 & 1.4e-02 \\ \hline
		9 & 10.0\% & 217.35 & 668 & 2e+06(789) & 36 & 1.5e-02 & 102.80 & 826 & 1e+06(954) & 31 & 1.6e-02 \\ \hline
		6 & 10.0\% & 259.70 & 961 & 2e+06(1126) & 38 & 1.5e-02 & 124.12 & 1109 & 1e+06(1276) & 38 & 1.3e-02 \\ \hline
		12 & 50.0\% & 337.82 & 558 & 3e+06(634) & 45 & 1.2e-02 & 83.17 & 676 & 1e+06(748) & 35 & 1.6e-02 \\ \hline
		9 & 50.0\% & 223.85 & 572 & 2e+06(667) & 53 & 8.6e-03 & 75.28 & 678 & 1e+06(765) & 61 & 7.0e-03 \\ \hline
		6 & 50.0\% & 194.21 & 668 & 1e+06(761) & 67 & 5.5e-03 & 76.30 & 684 & 7e+05(834) & 71 & 3.4e-03 \\ \hline
	\end{tabular}
\end{table}

We make the following comments on the results of the noisy case:
\begin{itemize}
	\item Compared to the noiseless case, it takes much more iterations for GAUGE and PFGAUGE
	to converge.
	\item The solutions produced by the two algorithms are similar in the aspect of
	the number of iterations, the successful recover rate, and the solution error.
	For most cases, PFGAUGE needs a little more iterations than GAUGE.
	\item PFGAUGE often provides 2-3 times speedup over GAUGE. This can be also told from the nDFT column.
	Again, the number of DFT is greatly reduced thanks to the polynomial filters.
\end{itemize}

Finally we conduct experiments on 2D real data in order to assess
the speedup of Chebyshev filters on problems with larger size.
In this scenario the measured signal $x_0$ are gray scale images,
summarized in Table \ref{tab:pl-real}. For simplicity,
the images are numbered from 1 to 12. Case 1 and 2 are selected from
the \textsc{matlab} image database; Case 3 to 12 are from
the HubbleSite Gallery\footnote{See \url{http://hubblesite.org/images/gallery}.}.
The largest problem (No. 12) has the size $n=4\cdot 10^6$, which
brings a huge eigenvalue problem in each iteration. For each example
we generate 10 and 15 octanary masks. The results are summarized
in Table \ref{tab:2Dreal}. The column with label ``$f$'' records
the values of dual objective function. The label ``gap'' stands
for the duality gap of the problem. The other columns have the same
meaning as in Table \ref{tab:noiseless-1024}. Since data size is very large,
we terminate the algorithms as soon as they exceed a timeout threshold,
which is set to 18000 seconds in the experiment.

\begin{table}[h]
	\centering
	\caption{Image data for 2D signals} \label{tab:pl-real}
	\begin{tabular}{|c|c|c||c|c|c|} \hline
		No. & name & size & No. & name & size \\ \hline
		1 & coloredChips & $518\times 391$ & 2 & lighthouse & $480\times 640$ \\ \hline
		3 & asteriods(S) & $500\times 500$ & 4 & giantbubble(S) & $600\times 570$ \\ \hline
		5 & supernova(S) & $640\times 426$ & 6 & nebula(S) & $800\times 675$ \\ \hline
		7 & crabnebula(S) & $1000\times 1000$ & 8 & asteriods(L) & $1000\times 1000$ \\ \hline
		9 & giantbubble(L) & $1200\times 1140$ & 10 & nebula(L) & $1600\times 1350$ \\ \hline
		%9 & giantbubble(L) & $1200\times 1140$ & 10 & supernova(L) & $1280 \times 853$ \\ \hline
		%11 & nebula(L) & $1600\times 1350$ & 12 & crabnebula(L) & $2000\times 2000$ \\ \hline
	\end{tabular}
\end{table}

\begin{table}[h]
	\centering
	\caption{Phase retrieval comparisons on 2D real signal} \label{tab:2Dreal}
	\setlength{\tabcolsep}{1.5pt}
	\begin{tabular}{|c|c|ccccc|ccccc|} \hline
		\multicolumn{2}{|c|}{ } & \multicolumn{5}{c|}{GAUGE} & \multicolumn{5}{c|}{PFGAUGE} \\ \hline
		No. & L & time & iter & nDFT(nEigs) & f & gap & time & iter & nDFT(nPF) & f & gap \\ \hline
		\multirow{2}{*}{1} & 15 & 1155.89 & 4 & 3e+05(19) & 1.0e+05 & 5.0e-06 & 549.97 & 7 & 2e+05(29) & 1.0e+05 & 1.4e-06\\ \cline{2-12}
		& 10 & 3704.66 & 9 & 1e+06(26) & 1.0e+05 & 1.0e-05 & 210.36 & 9 & 6e+04(26) & 1.0e+05 & 7.9e-06\\ \hline 
		\multirow{2}{*}{2} & 15 & 3043.72 & 4 & 1e+06(19) & 1.1e+05 & 5.7e-06 & 722.95 & 6 & 2e+05(26) & 1.1e+05 & 1.5e-06\\ \cline{2-12}
		& 10 & 21898.10 & 11 & 8e+06(87) & 1.1e+05 & NaN & 158.82 & 8 & 5e+04(25) & 1.1e+05 & 5.0e-06\\ \hline
		\multirow{2}{*}{3} & 15 & 276.71 & 4 & 8e+04(18) & 1.4e+04 & 4.9e-06 & 230.50 & 5 & 7e+04(21) & 1.4e+04 & 6.2e-06\\ \cline{2-12}
		& 10 & 385.09 & 13 & 1e+05(34) & 1.4e+04 & 9.5e-06 & 208.24 & 10 & 6e+04(28) & 1.4e+04 & 9.4e-06\\ \hline 
		\multirow{2}{*}{4} & 15 & 9583.33 & 12 & 3e+06(35) & 2.3e+04 & 1.3e-06 & 295.84 & 6 & 8e+04(26) & 2.3e+04 & 3.6e-06\\ \cline{2-12}
		& 10 & 7433.17 & 18 & 2e+06(46) & 2.3e+04 & 7.9e-06 & 238.08 & 7 & 6e+04(24) & 2.3e+04 & 9.7e-06 \\ \hline
		\multirow{2}{*}{5} & 15 & 1735.88 & 5 & 5e+05(20) & 2.2e+04 & 3.9e-06 & 622.25 & 9 & 2e+05(29) & 2.2e+04 & 7.5e-06\\ \cline{2-12}
		& 10 & 1872.62 & 6 & 5e+05(22) & 2.2e+04 & 1.5e-06 & 291.96 & 7 & 7e+04(24) & 2.2e+04 & 1.8e-06 \\ \hline
		\multirow{2}{*}{6} & 15 & 643.35 & 7 & 1e+05(24) & 6.7e+04 & 2.7e-06 & 523.80 & 9 & 8e+04(30) & 6.7e+04 & 6.4e-06 \\ \cline{2-12}
		& 10 & 21943.32 & 15 & 3e+06(42) & 6.7e+04 & 1.2e-02 & 989.89 & 4 & 1e+05(10) & 6.7e+04 & 6.7e-06 \\ \hline 
		\multirow{2}{*}{7} & 15 & 2252.18 & 10 & 2e+05(30) & 6.3e+04 & 8.3e-06 & 1116.42 & 9 & 1e+05(30) & 6.3e+04 & 7.6e-06 \\ \cline{2-12}
		& 10 & 2530.10 & 22 & 2e+05(54) & 6.3e+04 & 8.4e-06 & 981.21 & 17 & 9e+04(46) & 6.3e+04 & 6.9e-06 \\ \hline 
		\multirow{2}{*}{8} & 15 & 18543.56 & 5 & 1e+06(22) & 5.8e+04 & 1.0e-03 & 4050.70 & 9 & 3e+05(30) & 5.8e+04 & 5.8e-06\\ \cline{2-12}
		& 10 & 19371.82 & 11 & 1e+06(33) & 5.8e+04 & 1.0e-03 & 1209.17 & 11 & 9e+04(34) & 5.8e+04 & 8.9e-06\\ \hline 
		\multirow{2}{*}{9} & 15 & 20239.04 & 2 & 1e+06(17) & 9.1e+04 & 7.8e-02 & 7046.67 & 7 & 5e+05(27) & 9.1e+04 & 9.9e-06\\ \cline{2-12}
		& 10 & 19610.56 & 8 & 1e+06(19) & 9.1e+04 & 6.0e-01 & 3892.26 & 6 & 2e+05(14) & 9.1e+04 & 4.7e-06 \\ \hline
		%\multirow{2}{*}{10} & 15 & 20801.19 & 3 & 1e+06(18) & 8.9e+04 & 3.7e-02 & 5691.78 & 10 & 4e+05(32) & 8.9e+04 & 1.5e-06 \\ \cline{2-12}
		%& 10 & 14699.58 & 24 & 9e+05(60) & 8.9e+04 & 9.1e-06 & 18156.41 & 51 & 1e+06(778) & 8.9e+04 & 7.3e-01\\ \hline 
		\multirow{2}{*}{10} & 15 & 7680.27 & 10 & 3e+05(32) & 2.7e+05 & 4.1e-06 & 4287.95 & 9 & 2e+05(30) & 2.7e+05 & 5.8e-06 \\ \cline{2-12}
		& 10 & 21958.19 & 5 & 8e+05(31) & 2.7e+05 & 1.7e-01 & 4042.50 & 24 & 1e+05(58) & 2.7e+05 & 4.8e-06\\ \hline 
		%\multirow{2}{*}{12} & 15 & 24451.56 & 3 & 6e+05(28) & 2.5e+05 & 4.6e-01 & 18885.73 & 11 & 5e+05(168) & 2.5e+05 & 7.4e+00\\ \cline{2-12}
		%& 10 & 20725.83 & 4 & 5e+05(17) & 2.5e+05 & 1.7e-01 & 5720.32 & 24 & 1e+05(60) & 2.5e+05 & 5.1e-06\\ \hline 
	\end{tabular}
\end{table}

Table \ref{tab:2Dreal} actually shows the effectiveness of PFGAUGE:
\begin{itemize}
	\item For all test cases, PFGAUGE successfully converges within the time limit,
	while the original GAUGE method times out in most test cases with large data.
	GAUGE also fails in the Case 2 with $L=10$. After checking the output of the
	algorithm, the reason is that the eigen-solver fails to converge in some iterations.
	\item For the cases where both algorithm converge, PFGAUGE is able to provide over
	20 times speedup. This can be verified by the nDFT column of the two methods as well.
	The reason is that the traditional eigen-solver \texttt{eigs} is not scalable enough
	to deal with large problems. The convergence is very slow thus the accuracy of
	the solutions is not tolerable. However, our polynomial-filtered algorithm is
	able to extract a proper low-rank subspace that contains the desired eigenvectors
	quickly, hence the performance is much better.
	\item PFGAUGE sometimes needs fewer iterations than GAUGE. This is because \texttt{eigs}
	fails to converge during the iterations, thus GAUGE produces wrong updates.
\end{itemize}
\subsubsection{Blind Deconvolution}
Again, we use the blind deconvolution formulation in \cite{doi:10.1137/15M1034283}.
In this model, we measure the convolution of two signals
$s_{1}\in\bC^{m}$ and $s_{2}\in\bC^{m}$. Let
$B_{1}\in\bC^{m\times n_{1}}$ and
$B_{2}\in\bC^{m\times n_{2}}$ be two wavelet bases. The circular
convolution of the signals is defined by
\begin{align*}
b = s_{1}*s_{2}&= (B_1x_1)*(B_2x_2)
\\&= \cF^{-1}\diag\big((\cF B_1x_1)(\cF B_2x_2)^T\big)
\\&= \cF^{-1}\diag\big((\cF B_1)(x_1\overline{x}_2^*)(\overline{\cF B}_2)^*\big)
=: \cA(x_{1}\overline{x}_{2}^{*}),
\end{align*}
where $\cA$ is the non-symmetric linear map whose adjoint is
\[
\cA^{*}(y) := (\cF B_1)^*\mathrm{Diag}(\cF y)(\overline{\cF B_2}).
\]
The non-smooth part $R(x)$ defined as the indicator function of the set
\[
\{x~|~\iprod{b}{x} - \|x\|_* \geq 1 \}.
\]

In the experiments of the blind deconvolution, the GAUGE algorithm and its
polynomial-filtered version PFGAUGE are tested. The tolerance of feasibility and
duality gap is set to a moderate level, namely, $5\cdot 10^{-4}$ and $5\cdot 10^{-3}$.
We need to mention that there is a primal-dual refinement strategy in
the GAUGE method as mentioned before. It consists of two sub-problems: generate
a refined primal solution from a dual solution (pfd) and dual from primal (dfp).
In the ``dfp'' process, the algorithm has to solve a least-squared problem using
the gradient descent method. The ``dfp'' step is very slow in the experiments
and frequently leads to algorithm failure. Therefore, we disable this step since
it is optional.

The test data are shown in Table \ref{tab:bd-data}. There are eight pictures numbered
from 1 to 8. Case 1 is from the original GAUGE paper \cite{doi:10.1137/15M1034283}.
Case 2 to 4 are from the \textsc{matlab} gallery. Case 5 to 8 are downloaded
from the HubbleSite Gallery as mentioned before.
The results are demonstrated in Table \ref{tab:blind-deconvolution}. The columns with
label ``xErr1'' and ``xErr2'' list the relative errors $\|x_j - \hat{x}_j\|_2/\|x_j\|_2,j=1,2$,
where $\hat{x}_j$ stands for the solution returned by the algorithms. The column with label
``rErr'' contains the relative residual $\|b - \cA(\hat{x}_1\bar{\hat{x}}_2^*)\|_2 / \|b\|_2$.
The other columns share the same meaning with Table \ref{tab:2Dreal}.
\begin{table}[h]
	\centering
	\caption{Image data for blind deconvolution} \label{tab:bd-data}
	\begin{tabular}{|c|c|c||c|c|c|} \hline
		No. & name & size & No. & name & size \\ \hline
		1 & shapes & $256\times 256$ & 2 & cameraman & $256\times 256$ \\ \hline
		3 & rice & $256\times 256$ & 4 & lighthouse & $480\times 600$ \\ \hline
		5 & crabnebula512 & $512\times 512$ & 6 & crabnebula1024 & $1024\times 1024$ \\ \hline
		7 & mars & $1280\times 1280$ & 8 & macs & $1280 \times 1280$ \\ \hline
	\end{tabular}
\end{table}
\begin{table}[h]
	\centering
	\caption{Blind deconvolution comparisons} \label{tab:blind-deconvolution}
	\setlength{\tabcolsep}{1.5pt}
	{%\scriptsize
	\begin{tabular}{|c|cccccccc|} \hline
		\multirow{2}{*}{No.}& \multicolumn{8}{c|}{GAUGE(nodfp)} \\ \cline{2-9}
		& time & iter & nDFT(nEigs) & f & gap & rErr & xErr1 & xErr2 \\ \hline 
		1 & 13.10 & 15 & 4248(16) & 1.681e+00 & 1.7e-03 & 3.4e-04 & 9.1e-02 & 5.5e-01  \\ \hline  
		2 & 14.34 & 14 & 4816(15) & 1.683e+00 & 4.1e-03 & 4.1e-04 & 1.4e-01 & 5.4e-01  \\ \hline
		3 & 15.66 & 18 & 5440(19) & 1.672e+00 & 2.0e-03 & 3.0e-04 & 1.3e-01 & 5.4e-01  \\ \hline
		4 & 63.48 & 21 & 7028(22) & 1.683e+00 & 4.3e-03 & 2.5e-04 & 1.1e-01 & 5.4e-01  \\ \hline
		5 & 35.62 & 14 & 4552(15) & 1.689e+00 & 4.7e-03 & 2.8e-04 & 1.3e-01 & 5.3e-01  \\ \hline
		6 & 252.59 & 26 & 9588(27) & 1.673e+00 & 3.0e-04 & 3.6e-04 & 1.0e-01 & 5.4e-01 \\ \hline
		7 & 282.38 & 22 & 7116(23) & 1.696e+00 & 4.7e-03 & 2.6e-04 & 9.8e-02 & 5.4e-01 \\ \hline
		8 & 559.64 & 36 & 14732(37) & 1.667e+00 & 4.7e-03 & 7.2e-05 & 9.4e-02 & 5.3e-01\\ \hline
		\multirow{2}{*}{No.}& \multicolumn{8}{c|}{PFGAUGE(nodfp)} \\ \cline{2-9}
		& time & iter & nDFT(nPF) & f & gap & rErr & xErr1 & xErr2 \\ \hline
		1 & 6.85 & 15 & 2740(16) & 1.681e+00 & 1.7e-03 & 4.0e-04 & 9.1e-02 & 5.5e-01   \\ \hline
        2 & 7.30 & 14 & 2844(15) & 1.684e+00 & 4.7e-03 & 5.1e-04 & 1.4e-01 & 5.4e-01   \\ \hline
        3 & 8.93 & 18 & 3576(19) & 1.671e+00 & 2.1e-04 & 3.7e-04 & 1.3e-01 & 5.4e-01   \\ \hline
        4 & 40.60 & 26 & 4560(27) & 1.683e+00 & 4.1e-03 & 4.1e-04 & 1.1e-01 & 5.4e-01  \\ \hline
        5 & 21.91 & 14 & 2800(15) & 1.689e+00 & 4.8e-03 & 2.9e-04 & 1.3e-01 & 5.3e-01  \\ \hline
        6 & 119.43 & 26 & 4524(27) & 1.677e+00 & 2.9e-03 & 3.2e-04 & 1.0e-01 & 5.4e-01 \\ \hline
        7 & 163.91 & 25 & 4204(26) & 1.696e+00 & 4.4e-03 & 3.0e-04 & 1.0e-01 & 5.4e-01 \\ \hline
        8 & 297.61 & 37 & 7568(38) & 1.659e+00 & 4.1e-04 & 3.4e-04 & 9.3e-02 & 5.3e-01 \\ \hline

	\end{tabular}}
\end{table}
We have similar observations under these settings. PFGAUGE is able to produce similar results
as GAUGE with two times speedup. The performance is not as impressive as what is shown in Table \ref{tab:2Dreal}
because the eigenvalue sub-problem is easier to solve in this case.
\subsection{ADMM for SDP}
In this section we show the numerical results of our PFAM.
The experiments fall into two categories: standard SDP
and non-linear SDP.
\subsubsection{Standard SDP}
First we test PFAM on
the two-body reduced density matrix (2-RDM) problem. The problem
has a block diagonal structure with respect to the variable $X$,
with each block being a low rank matrix. Thus the polynomial filters
can be applied to each block to reduce the cost. We use the preprocessed
dataset in \cite{li2017semi}. Only the cases with the maximum block size
greater than 1000 are selected. The details of the data can be found in
\cite {nakata2008variational}.

Table \ref{tab:sdp-standard} contains the numerical results of ADMM and PFAM.
The columns with headers ``pobj'', ``pinf'', ``dinf'', and ``gap'' record the primal
objective function value, primal infeasibility, dual infeasibility, and duality gap at the solution,
respectively. The overall tolerance of each algorithm is set to $10^{-4}$.
A high accuracy is not needed in this application because we use ADMM to generate
initial solutions for second-order methods such as SSNSDP \cite{li2017semi}.

\begin{table}[h]
	\caption{Results of ADMM and PFAM on 2-RDM problems} \label{tab:sdp-standard}
	\centering
	\setlength{\tabcolsep}{2pt}
	\begin{tabular}{|c|cccccc|cccccc|}
		\hline 
		\multicolumn{1}{|c|}{data} & \multicolumn{6}{c|}{ADMM} & \multicolumn{6}{c|}{PFAM} \\ \hline 
		name & time & itr & pobj & pinf & dinf & gap & time & itr & pobj & pinf & dinf & gap \\ \hline
		AlH &     73 &     373 &   2.46e+02 &  2.1e-05 &  9.9e-05 &  4.8e-05 &     71 &     377 &   2.46e+02 &  2.1e-05 &  9.9e-05 &  4.7e-05 \\ 
		B2 &    428 &     800 &   5.73e+01 &  9.6e-05 &  9.9e-05 &  1.2e-04 &    257 &     908 &   5.73e+01 &  8.1e-05 &  9.9e-05 &  1.3e-04 \\ 
		BF &     64 &     339 &   1.42e+02 &  3.6e-05 &  9.9e-05 &  7.0e-05 &     62 &     341 &   1.42e+02 &  3.6e-05 &  9.9e-05 &  6.9e-05 \\ 
		BH &    476 &    1058 &   2.73e+01 &  9.6e-05 &  9.5e-05 &  3.7e-04 &    322 &    1026 &   2.73e+01 &  9.1e-05 &  9.8e-05 &  3.5e-04 \\ 
		BH3O &    727 &     437 &   1.31e+02 &  6.9e-05 &  9.9e-05 &  2.0e-04 &    530 &     458 &   1.31e+02 &  2.8e-05 &  9.9e-05 &  1.7e-04 \\ 
		BN &    144 &     748 &   9.33e+01 &  3.0e-05 &  9.9e-05 &  1.1e-04 &     93 &     739 &   9.33e+01 &  4.2e-05 &  9.9e-05 &  1.2e-04 \\ 
		BO &     91 &     377 &   1.17e+02 &  8.8e-05 &  9.9e-05 &  9.9e-05 &     80 &     375 &   1.17e+02 &  9.8e-05 &  9.2e-05 &  1.0e-04 \\ 
		BeF &     68 &     356 &   1.28e+02 &  5.3e-05 &  9.9e-05 &  9.0e-05 &     58 &     354 &   1.28e+02 &  5.5e-05 &  9.8e-05 &  9.1e-05 \\ 
		BeO &     91 &     488 &   1.02e+02 &  7.3e-05 &  9.9e-05 &  8.8e-05 &     71 &     479 &   1.02e+02 &  4.8e-05 &  9.9e-05 &  9.6e-05 \\ 
		C2 &   1763 &     831 &   9.10e+01 &  5.8e-05 &  9.9e-05 &  1.8e-04 &   1032 &     911 &   9.10e+01 &  5.2e-05 &  9.9e-05 &  1.7e-04 \\ 
		CF &     68 &     354 &   1.59e+02 &  3.5e-05 &  9.9e-05 &  3.6e-05 &     60 &     351 &   1.59e+02 &  3.8e-05 &  9.9e-05 &  3.4e-05 \\ 
		CH &    453 &     969 &   4.12e+01 &  9.7e-05 &  9.8e-05 &  3.6e-04 &    303 &    1009 &   4.12e+01 &  9.8e-05 &  9.6e-05 &  3.4e-04 \\ 
		CH2 &   1516 &    1350 &   4.50e+01 &  8.0e-05 &  9.9e-05 &  3.9e-04 &    754 &    1339 &   4.50e+01 &  8.3e-05 &  9.9e-05 &  3.6e-04 \\ 
		CH2 &   1698 &    1496 &   4.48e+01 &  9.8e-05 &  9.2e-05 &  3.4e-04 &    782 &    1428 &   4.48e+01 &  9.8e-05 &  9.5e-05 &  3.5e-04 \\ 
		CH3 &   2028 &    1113 &   4.94e+01 &  5.5e-05 &  9.9e-05 &  1.4e-04 &   1010 &    1149 &   4.94e+01 &  6.7e-05 &  9.9e-05 &  2.5e-04 \\ 
		CH3N &    756 &     450 &   1.27e+02 &  4.8e-05 &  9.9e-05 &  1.6e-04 &    491 &     456 &   1.27e+02 &  5.9e-05 &  9.9e-05 &  1.5e-04 \\ 
		CN &     83 &     439 &   1.11e+02 &  8.9e-05 &  9.9e-05 &  8.8e-05 &     68 &     437 &   1.11e+02 &  9.8e-05 &  9.5e-05 &  8.3e-05 \\ 
		CO+ &     73 &     379 &   1.35e+02 &  8.9e-05 &  9.9e-05 &  9.1e-05 &     60 &     377 &   1.35e+02 &  9.4e-05 &  9.9e-05 &  9.1e-05 \\ 
		CO &     63 &     328 &   1.35e+02 &  6.7e-05 &  9.8e-05 &  3.5e-05 &     54 &     325 &   1.35e+02 &  6.8e-05 &  9.9e-05 &  3.4e-05 \\ 
		F- &    265 &     648 &   9.96e+01 &  9.9e-05 &  9.5e-05 &  2.9e-04 &    189 &     639 &   9.96e+01 &  9.0e-05 &  9.9e-05 &  2.4e-04 \\ 
		H2O &    819 &     717 &   8.54e+01 &  9.5e-05 &  9.9e-05 &  3.4e-04 &    495 &     809 &   8.54e+01 &  9.8e-05 &  9.0e-05 &  3.4e-04 \\ 
		HF &    268 &     576 &   1.05e+02 &  9.9e-05 &  9.3e-05 &  2.2e-04 &    210 &     575 &   1.05e+02 &  9.6e-05 &  9.8e-05 &  2.1e-04 \\ 
		HLi2 &    189 &     622 &   1.91e+01 &  7.5e-05 &  9.8e-05 &  1.4e-04 &    126 &     656 &   1.91e+01 &  7.9e-05 &  9.8e-05 &  1.5e-04 \\ 
		HN2+ &     99 &     326 &   1.38e+02 &  7.7e-05 &  9.9e-05 &  7.3e-05 &     82 &     326 &   1.38e+02 &  8.7e-05 &  9.9e-05 &  7.1e-05 \\ 
		HNO &    251 &     451 &   1.60e+02 &  7.4e-05 &  9.9e-05 &  4.1e-05 &    207 &     446 &   1.60e+02 &  7.8e-05 &  9.9e-05 &  3.9e-05 \\ 
		LiF &     80 &     406 &   1.16e+02 &  4.4e-05 &  9.9e-05 &  9.9e-05 &     73 &     406 &   1.16e+02 &  4.5e-05 &  9.9e-05 &  9.8e-05 \\ 
		LiH &    472 &    1030 &   9.00e+00 &  9.0e-05 &  4.2e-05 &  1.1e-04 &    369 &    1037 &   9.00e+00 &  9.1e-05 &  5.4e-05 &  1.2e-04 \\ 
		LiOH &    141 &     468 &   9.57e+01 &  9.8e-05 &  8.5e-05 &  1.4e-04 &    107 &     452 &   9.57e+01 &  9.5e-05 &  9.6e-05 &  1.6e-04 \\ 
		NH &    443 &     943 &   5.86e+01 &  8.6e-05 &  9.9e-05 &  3.0e-04 &    298 &     946 &   5.86e+01 &  8.8e-05 &  9.9e-05 &  3.1e-04 \\ 
		NH &    405 &     870 &   5.86e+01 &  8.9e-05 &  9.9e-05 &  2.8e-04 &    290 &     937 &   5.86e+01 &  9.8e-05 &  9.1e-05 &  2.7e-04 \\ 
		NH2- &    967 &     849 &   6.32e+01 &  9.8e-05 &  9.3e-05 &  4.0e-04 &    502 &     854 &   6.32e+01 &  9.7e-05 &  9.7e-05 &  3.9e-04 \\ 
		NH3 &   3775 &     925 &   6.82e+01 &  7.3e-05 &  9.9e-05 &  3.2e-04 &   1819 &     967 &   6.82e+01 &  9.7e-05 &  9.8e-05 &  3.4e-04 \\ 
		NaH &     73 &     370 &   1.65e+02 &  8.9e-05 &  9.8e-05 &  1.7e-06 &     66 &     370 &   1.65e+02 &  8.9e-05 &  9.9e-05 &  5.0e-06 \\ 
		P &    176 &     410 &   3.41e+02 &  2.2e-05 &  9.9e-05 &  7.0e-05 &    155 &     412 &   3.41e+02 &  2.2e-05 &  9.8e-05 &  6.6e-05 \\ 
		SiH4 &    212 &     292 &   3.12e+02 &  2.8e-05 &  9.9e-05 &  3.8e-05 &    193 &     292 &   3.12e+02 &  2.8e-05 &  9.9e-05 &  2.8e-05 \\
	    \hline
	\end{tabular}
\end{table}

As observed in Table \ref{tab:sdp-standard}, PFAM is two times faster than ADMM in large 2-RDM
problems, such as CH2, C2, and NH3. We mention that the speedup provided by PFAM depends
on the number of large low-rank variable blocks. One may notice that PFAM is less
effective on some problems like AlH and BF. The main reason is that there are very
few large blocks in these dataset, and polynomial filters are not designed for
small matrices. In fact, it is always observed that a full eigenvalue
decomposition is faster than any other truncated eigen-solvers or polynomial-filtered methods
in these small cases. Another minor reason is that some blocks is not low-rank, 
causing the performance of PFAM to be limited. This observation again addresses
the feature of PFAM: it is the most suitable for large-scale low-rank problems.

\subsubsection{Non-linear SDP}
As an extension, we consider plugging polynomial filters into multi-block ADMM to
observe its speed-up. We are interested in the non-linear SDPs from the weighted LS model with spectral norm constraints
and least unsquared deviations (LUD) model in \cite{wang2013orientation}.

Suppose $K$ is a given integer and $S$ and $W$ are two known matrices,
the weighted LS model with spectral norm constraints is
\begin{equation} \label{eqn:LS-spectral}
\begin{array}{rl}
\max & \iprod{W\odot S}{G}, \\
\mathrm{s.t.} & G_{ii} = I_2, \\
& G\succeq 0, \\
& \|G\|_2 \leq \alpha K,
\end{array}
\end{equation}
where $G = (G_{ij})_{i,j=1,\ldots,K} \in \mathcal{S}^{2K}$ is the variable, with each block $G_{ij}$ being
a 2-by-2 small matrix, and $\|\cdot\|_2$ is the spectral norm. A three-block ADMM is introduced to
solve \eqref{eqn:LS-spectral}:
\begin{eqnarray}
y^{k+1} &=& -\cA(C + X^k - Z^k) - \frac{1}{\mu}(\cA(G^k) - b), \label{eqn:ADMM-LS-y}\\
Z^{k+1} &=& \argmin_{Z} \frac{\alpha K}{\mu}\|Z\|_* + \frac{1}{2}\|Z - B^k\|_F^2, \label{eqn:ADMM-LS-Z}
\end{eqnarray}
\begin{eqnarray}
X^{k+1} &=& \argmin_{X\succeq 0} \|X - H^k\|_F^2, \label{eqn:ADMM-LS-X}\\
G^{k+1} &=& (1 - \gamma)G^k + \gamma \mu(X^{k+1} - H^k) \label{eqn:ADMM-LS-G}, 
\end{eqnarray}
where 
\begin{equation*}
\begin{split}
B^k &= C + X^k + \cA^*(y^{k+1}) + \frac{1}{\mu}G^k, \\
H^k &= Z^{k+1} - C - \mathcal{A}^*(y^{k+1}) - \frac{1}{\mu} G^k 
\end{split}
\end{equation*}
are the auxiliary variables. In the ADMM update, steps \eqref{eqn:ADMM-LS-Z} and \eqref{eqn:ADMM-LS-X}
require EVDs, both of which can be replaced with the polynomial-filtered
method.

The semi-definite relaxation of the LUD problem is
\begin{equation} \label{eqn:LUD}
\begin{array}{rl}
\min & \sum_{1\leq i < j \leq K} \|c_{ij} - G_{ij}c_{ji}\|_2, \\
\mathrm{s.t.} & G_{ii} = I_2, \\
& G \succeq 0, \\
& \|G\|_2 \leq \alpha K, \\
\end{array}
\end{equation}
where $G$, $G_{ij}$, $K$ are defined the same in \eqref{eqn:LS-spectral}, and $c_{ij} \in \mathbb{R}^2$
are known vectors. The spectral norm constraint in \eqref{eqn:LUD} is optional. The authors proposed
a four-block ADMM to solve \eqref{eqn:LUD}. The update scheme is quite similar with
\eqref{eqn:ADMM-LS-y}-\eqref{eqn:ADMM-LS-G}. We omit the full ADMM updates in this section. One can refer
to \cite{wang2013orientation} for more details.

We generate simulated data as the authors did in \cite{wang2013orientation}. First $K$ centered
images of $129\times 129$ with different orientations are generated. The $K$ rotation matrices
$R_i,i=1,\ldots,K$ is uniformly sampled. For simplicity no noise is add to the test images. Then we solve model
\eqref{eqn:LS-spectral} and \eqref{eqn:LUD} to obtain the Gram matrix $\hat{G}$. Finally the estimated
orientation matrices $\hat{R}_i$ are extracted from $\hat{G}$. The mean squared error defined in
\eqref{eqn:orientation-MSE} is used to evaluate the accuracy of $\hat{R}_i$.
\begin{equation} \label{eqn:orientation-MSE}
MSE = \frac{1}{K}\sum_{i=1}^{K}\|R_i - \hat{O}\hat{R}_i\|_F^2.
\end{equation}
The matrix $\hat{O}$ is the optimal registration matrix between $R_i,i=1,\ldots,K$ and
$\hat{R}_i,i=1,\ldots,K$.
Table \ref{tab:sdp-non-linear} shows the results of ADMM and PFAM on various settings.
Here $\alpha = 0.00$ means that there is no spectral norm constraint in \eqref{eqn:LUD}.
The number of the rotation matrices $K$ is chosen from $500,1000,1500,2000$. The primal
infeasibility (pinf), dual infeasibility (dinf), and mean squared error (mse) are
also recorded.

\begin{table}[h]
	\centering
	\caption{Results of ADMM and PFAM on non-linear SDPs} \label{tab:sdp-non-linear}
	\begin{tabular}{|c|ccccc|ccccc|}
	\hline 
	\multicolumn{11}{|c|}{LS model, $\alpha = 0.67$} \\ \hline
	\multirow{2}{*}{$K$} & \multicolumn{5}{c|}{ADMM} & \multicolumn{5}{c|}{PFAM} \\ \cline{2-11} 
	& time & itr & pinf & dinf & mse & time & itr & pinf & dinf & mse \\ \hline
	  500 &   36.1 &  232 &  9.9e-04 &  1.9e-04 &  1.46e-02 &    9.2 &  231 &  1.0e-03 &  1.9e-04 &  1.46e-02 \\ 
	 1000 &  120.1 &  163 &  9.7e-04 &  6.9e-04 &  9.82e-03 &   21.2 &  163 &  9.7e-04 &  6.9e-04 &  9.82e-03 \\ 
	 1500 &  426.6 &  189 &  9.9e-04 &  3.2e-04 &  6.07e-03 &   74.4 &  189 &  9.9e-04 &  3.2e-04 &  6.07e-03 \\ 
	 2000 & 1189.6 &  202 &  9.9e-04 &  1.7e-04 &  4.42e-03 &  148.7 &  202 &  9.9e-04 &  1.7e-04 &  4.42e-03 \\ 
	\hline
	\hline 
	\multicolumn{11}{|c|}{LUD model, $\alpha = 0.00$} \\ \hline
	\multirow{2}{*}{$K$} & \multicolumn{5}{c|}{ADMM} & \multicolumn{5}{c|}{PFAM} \\ \cline{2-11} 
	& time & itr & pinf & dinf & mse & time & itr & pinf & dinf & mse \\ \hline
	  500 &    7.9 &   78 &  9.9e-04 &  8.4e-05 &  9.99e-03 &    3.8 &   83 &  9.6e-04 &  8.2e-05 &  9.99e-03 \\ 
	 1000 &   44.4 &   99 &  9.4e-04 &  9.5e-04 &  3.06e-03 &   15.3 &   90 &  9.8e-04 &  4.6e-05 &  5.42e-03 \\ 
	 1500 &  131.9 &  101 &  9.1e-04 &  6.1e-04 &  2.24e-03 &   47.8 &  118 &  9.1e-04 &  6.5e-04 &  2.24e-03 \\ 
	 2000 &  334.5 &  102 &  9.3e-04 &  4.1e-04 &  1.91e-03 &   79.0 &  103 &  9.4e-04 &  4.2e-04 &  1.91e-03 \\ 
	\hline
	\hline 
	\multicolumn{11}{|c|}{LUD model, $\alpha = 0.67$} \\ \hline
	\multirow{2}{*}{$K$} & \multicolumn{5}{c|}{ADMM} & \multicolumn{5}{c|}{PFAM} \\ \cline{2-11} 
	& time & itr & pinf & dinf & mse & time & itr & pinf & dinf & mse \\ \hline
	  500 &   47.2 &  274 &  1.0e-03 &  2.7e-04 &  1.91e-03 &   14.9 &  277 &  9.9e-04 &  2.7e-04 &  1.91e-03 \\ 
	 1000 &  294.3 &  356 &  1.0e-03 &  1.3e-04 &  1.74e-03 &   68.0 &  369 &  1.0e-03 &  1.3e-04 &  1.74e-03 \\ 
	 1500 &  871.8 &  364 &  1.0e-03 &  1.2e-04 &  1.60e-03 &  165.7 &  316 &  1.0e-03 &  4.4e-04 &  1.63e-03 \\ 
	 2000 & 2526.5 &  413 &  1.0e-03 &  3.2e-04 &  2.41e-03 &  347.2 &  373 &  1.0e-03 &  4.1e-04 &  2.41e-03 \\ 
	\hline
	\end{tabular}
\end{table}

We observe that PFAM requires similar iterations but provides 3 to 9 times speedup over ADMM.
For these two problems, the solution $G$ is low-rank. Indeed it is a rank-3 matrix in
the form $RR^T$. Note that for ADMM, we have tested both full eigenvalue decomposition and
the truncated version and report the one which costs less time. These examples again
justify the effectiveness of PFAM, though the convergence
property is not clear for the multi-block version.

\section{Conclusion} \label{sec:conclusion}
In this paper, we propose a framework of polynomial-filtered
methods for low-rank optimization problems. Our motivation is based on
the key observation that the iteration points lie in a low-rank subspace
of $\bR^{n\times n}$. Therefore, the strategy is to extract this subspace
approximately, and then perform one update based on the projection
of the current iteration point. Polynomials are also applied to increase
the accuracy. Intuitively, the target subspaces
between any two iterations should be close enough under some conditions. 
We next give two solid examples PFPG and PFAM in order to
show the basic structure of polynomial-filtered methods. It is easy to
observe that this kind of method couples the subspace refinement and the main
iteration together.
In the theoretical part, we  analyze the convergence of PFPG and PFAM.
A key assumption is that the initial subspace should not be orthogonal to
the target subspace to be used in the next iteration. Together
with the Chebyshev polynomials we are able to estimate the approximation error
of the subspace. The main convergence result indicates that the degree of
the polynomial can remain a constant as the iterations proceed, which is meaningful
in real applications. Even if the warm-start property is not considered, the
degree grows very slowly (about order $\cO(\log k)$) to ensure the convergence.
Our numerical 
experiments shows that the polynomial-filtered algorithms are pretty effective
on low-rank problems compared to the original methods, since they successfully
reduce the computational costs of large-scale EVDs. Meanwhile, the number of iterations
is barely increased. These observations coincide with our theoretical results.

%It will be interesting to carry out a uniform convergence analysis on
%any polynomial-filtered method as long as the original one can be
%proved to converge. Polynomial-filtered quasi-Newton methods
%also seem promising. These topics can be left for our future work.
\bibliographystyle{siam}
\bibliography{spectual}
\end{document}